\documentclass{amsart}
\usepackage{graphicx}
\usepackage{amssymb}
\usepackage{amsfonts}
\usepackage{hyperref}
\usepackage{mathrsfs}
\swapnumbers
\newtheorem{theorem}{Theorem}[section]
\newtheorem{lemma}[theorem]{Lemma}

\newtheorem{proposition}[theorem]{Proposition}
\theoremstyle{definition}

\newtheorem{assumption}[theorem]{Assumption}
\newtheorem{remark}[theorem]{Remark}

\numberwithin{equation}{section}
\theoremstyle{plain}

\numberwithin{equation}{section} 
\numberwithin{figure}{section} 
\theoremstyle{plain}
\theoremstyle{plain}
\theoremstyle{remark}
\newtheorem*{acknowledgement*}{Acknowledgement}
\theoremstyle{example}

\newcommand{\cA}{{\mathcal A}}
\newcommand{\cB}{{\mathcal B}}
\newcommand{\cC}{{\mathcal C}}
\newcommand{\cD}{{\mathcal D}}
\newcommand{\cE}{{\mathcal E}}
\newcommand{\cF}{{\mathcal F}}

\newcommand{\cH}{{\mathcal H}}

\newcommand{\cJ}{{\mathcal J}}
\newcommand{\cK}{{\mathcal K}}
\newcommand{\cL}{{\mathcal L}}

\newcommand{\cO}{{\mathcal O}}
\newcommand{\cP}{{\mathcal P}}

\newcommand{\cS}{{\mathcal S}}
\newcommand{\cT}{{\mathcal T}}

\newcommand{\cX}{{\mathcal X}}
\newcommand{\cY}{{\mathcal Y}}

\newcommand{\te}{{\theta}}

\newcommand{\Om}{{\Omega}}
\newcommand{\om}{{\omega}}
\newcommand{\ve}{{\varepsilon}}
\newcommand{\del}{{\delta}}
\newcommand{\Del}{{\Delta}}
\newcommand{\gam}{{\gamma}}
\newcommand{\Gam}{{\Gamma}}

\newcommand{\sig}{{\sigma}}
\newcommand{\al}{{\alpha}}
\newcommand{\be}{{\beta}}
\newcommand{\ka}{{\kappa}}
\newcommand{\la}{{\lambda}}

\newcommand{\bbC}{{\mathbb C}}
\newcommand{\bbE}{{\mathbb E}}

\newcommand{\bbN}{{\mathbb N}}

\newcommand{\bbR}{{\mathbb R}}

\newcommand{\bbZ}{{\mathbb Z}}
\newcommand{\bbI}{{\mathbb I}}

\begin{document}
\title[]{Limit theorems for some skew products with mixing base maps} 
 \vskip 0.1cm
 \author{Yeor Hafouta \\
\vskip 0.1cm
 Institute  of Mathematics\\
Hebrew University\\
Jerusalem, Israel}%
\address{
Institute of Mathematics, The Hebrew University, Jerusalem 91904, Israel}
\email{yeor.hafouta@mail.huji.ac.il}%

\thanks{ }
\dedicatory{  }
 \date{\today}

\begin{abstract}\noindent
We obtain central limit theorem, local limit theorems and renewal theorems for 
stationary processes generated by skew product maps $T(\om,x)=(\te\om,T_\om x)$ together with a $T$-invariant measure,
whose base map $\te$ satisfies certain topological and mixing conditions and the
maps $T_\om$ on the fibers are certain non-singular distance expanding maps. Our results hold true when $\te$ is either a sufficiently fast mixing Markov shift with positive transition densities or a
(non-uniform) Young tower with at least one periodic point 
and polynomial tails.  
The proofs are based on the random complex Ruelle-Perron-Frobenius theorem from \cite{book}
applied with appropriate random 
transfer operators generated by $T_\om$, together with certain regularity assumptions 
(as functions of $\om$) of these operators. 
Limit theorems for deterministic processes whose distributions on the fibers are generated by 
Markov chains with transition operators satisfying a random version of the Doeblin condition 
will also be obtained.
The main innovation in this paper is that the results hold true even though the spectral theory used in \cite{Aimino} does not seem to 
be applicable, and the dual of the Koopman operator of $T$ 
(with respect to the invariant measure) does not seem to have a spectral gap.

\end{abstract}
\maketitle
\markboth{Y. Hafouta}{Limit theorems some skew products with mixing invertible base maps} 
\renewcommand{\theequation}{\arabic{section}.\arabic{equation}}
\pagenumbering{arabic}

\section{Introduction}\label{sec1}\setcounter{equation}{0}
Probabilistic limit theorems for deterministic dynamical systems and Markov chains
is a well studied topic. Many results of this type are consequences of quasi-compactness (i.e. spectral gap) of an appropriate Markov operator together with some
perturbation theorem (see, for instance, \cite{GH}, \cite{HenHerve}, \cite{Neg1} and \cite{Neg2}).
The situation with limit theorems for random dynamical systems and Markov 
chains in random environment is more complicated, since, as opposed to the deterministic case, 
there is not only one underlying operator, but a family of random operators, 
so no spectral theory can be exploited. The central limit theorem (CLT) and large deviations theorem 
in this context can be obtained (see \cite{Kifer-1998}, \cite{Kifer-1991}, \cite{Kifer-1996} and references therein)  
from the random real Ruelle-Perron-Frobenius (RPF) theorem (see, for instance, \cite{Kifer-1996} 
and \cite{Kifer-Thermo}). Relying on this RPF theorem, limit theorems also follow from the spectral approach
of \cite{drag}.
In \cite{book} we proved a random complex RPF theorem and used it to obtain (under certain conditions) a version of the Berry-Esseen theorem and the local central limit theorem for such processes in random environments.

Let $(\Om,\cF,P,\te)$ be an ergodic invertible measure preserving system (MPS), $\cX$ be a compact space and $T_\om:\cX\to\cX$ be a random expanding non-singular transformation with respect to some probability measure
$\textbf{m}$. In this paper, under certain conditions, we will prove (annealed) limit theorems for stationary processes generated by the skew product map $T(\om,x)=(\te\om,T_\om x)$, together with a $T$-invariant measure $\boldsymbol{\mu}=\int\boldsymbol{\mu}_\om dP(\om)$ whose disintegrations $\boldsymbol{\mu}_\om$ are certain random Gibbs measures. We will also obtain limit theorems for non-invertible MPS's in the case when $T$ preserves a product measure of the form $\boldsymbol{\mu}=P\times(\bar h\textbf{m})$ for some continuous function $\bar h$. In \cite{Aimino} (see also references therein), the authors proved annealed limit theorems in the case when $T_\om,T_{\te\om},T_{\te^2\om},...$ are independent and the skew product map preserves a measure of the above product form.
Our results hold true when the compositions $T_{\te^{n-1}\om}\circ T_{\te^{n-2}\om}\circ\cdots\circ T_{\om}$ are taken along orbits $\{\te^k\om:\,k\geq0\}$ of a map $\te$ satisfying some mixing and topological conditions and having at least one periodic point, 
assuming that appropriate perturbations $\cL_{it}^\om,\,t\in\bbR$ of the dual $\cL_0^\om$ (with respect to $\textbf{m}$) of the Koopman operator $g\to g\circ T_\om$
satisfy some regularity conditions (as functions of $\om$) around one periodic orbit of $\te$.

Quenched limit theorems, in our context, describe the asymptotic behaviour of the distribution the iterates 
$T_{\te^{n-1}\om}\circ T_{\te^{n-2}\om}\circ\cdots\circ T_{\om}x,\,n\geq1$, when $\om$ is fixed (but rages over a set of full $P$-probability) and $x$ is distributed 
according to a random Gibbs measure.
Often  annealed  limit theorems follow almost directly from the quenched ones just by integration over appropriate sets $\Gam\subset\Om$. We will use this approach (see Section \ref{sec4}) to prove an annealed CLT, but, for general $\te$'s, finer (annealed) limit theorems do not follow directly from the corresponding quenched limit theorems. For instance, in the local central limit theorem (LLT) the asymptotic behaviour as $n\to\infty$ of certain expectations multiplied by $\sqrt n$ is studied (see Theorem \ref{LLT--D.S.}), which makes it impossible to choose appropriate
sets $\Gam$ which do not depend on $n$. The (annealed) renewal theorem also does not follow from a corresponding quenched limit theorem,
since (see Theorem \ref{renewal-D.S.}) it describes the asymptotic behaviour of certain series of expectations, 
and either way, to the best of my knowledge, no quenched renewal theorem has been proved in the setup of this paper.

In several circumstances annealed limit theorems follow from spectral theory of a single Markov operator
together with some perturbation theorem.
For instance, the results in \cite{Aimino} and \cite{Ish} rely on the theory of quasi compact operators, where due to independence of the random maps the authors could exploit spectral properties of the averaged (Fourier) operators
$\cA_{it}=\int\cA_{it}^\om dP(\om)$, obtained from integrating appropriate perturbations $\cA_{it}^\om$
of the dual $\cA_0^\om$ of the Koopman operator $g\to g\circ T_\om$ (with respect to an appropriate  measure). 
In our situation the maps $T_\om, T_{\te\om},T_{\te^2\om},...$ are not independent, so there is no connection between the iterates of the averaged Fourier operators and the average $\int \cA_{it}^{\om,n}dP(\om)$ of the iterates $\cA_{it}^{\om,n}=\cA_{it}^{\te^{n-1}\om}\circ\cdots\circ \cA_{it}^{\te\om}\circ A_{it}^\om$ of the random Fourier operators. 
Another example is the case when the map $\te$ is distance expanding. In this case the skew product map $T$ is also distance expanding (since $T_\om$'s are), and so limit theorems for stationary sums generated by $T$ follow from the spectral theory of the dual $\cL$ of the Koopman
operator $g\to g\circ T$. For instance, $\te$ can be a topologically mixing subshift of finite type (see \cite{Bow}) or a uniform Young tower (see \cite{Young1} and \cite{Chaz}), but our results hold true for certain non-distance expanding maps such as (non-uniform) Young towers (see \cite{Young2} and \cite{Chaz}) and uncountable Markov shifts.

In order to overcome the inapplicability of the spectral theory
described in the previous paragraph, we will first apply the random complex RPF theorem from \cite{book}.
After applying this RPF theorem, the main difficulty in proving the LLT and renewal theorem
arises in obtaining appropriate (at least polynomial) 
decay of the integrals
$\int \|\cA_{it}^{\om,n}\|dP(\om),\,t\not=0$ as $n\to0$. When $\om$ is fixed, we obtained 
in \cite{book} certain quenched LLT and Berry-Esseen theorem by 
showing that the operator norms $\|\cA_{it}^{\om,n}\|,\,t\not=0$ decay appropriately to $0$ as $n\to\infty$. 
The problem in obtaining corresponding estimates of $\int \|\cA_{it}^{\om,n}\|dP(\om)$ by integration
is that the rate of convergence of $\|\cA_{it}^{\om,n}\|,\,t\not=0$ to $0$, in general, is not uniform in $\om$
as it depends, among other things, on certain almost sure limit theorems (e.g. on ergodic theorems). 
We will show that under appropriate mixing and topological conditions we can control, uniformly for $t$'s belonging to compact sets $J$ not containing the origin, the order in $n$ of $\|\cA_{it}^{\om,n}\|$  on sets $\Gam_n=\Gam_n(J)$, so that $1-P(\Gam_n)$ decays polynomially fast to $0$. The arguments in the proof of these estimates can be viewed as  
averaged ("annealed") version of the periodic point approach from \cite{book} (see Section 2.10 and Chapter 7 there).

Our results  hold true when $\te$ is the two sided Markov shift generated by a sufficiently fast mixing stationary Markov chain $\zeta_n,\,n\geq0$ with positive transition densities and initial distribution assigning positive mass to open sets. In this case
$\Om=\cY^\bbZ$, where $\cY$ is the state space of the chain, and periodic points have the form $\bar a=(...,a,a,a,...)$, $\bar a=(a_i)_{i=0}^{n_0-1}\in\cY^{n_0}$.
In fact, our conditions will also be satisfied when the shift is generated by a stationary and sufficiently fast mixing
process so that for some periodic $\bar a$ we have 
$P(\zeta_{i+(j-1)n_0}\in A_{j,i};\,0\leq i<n_0,\,1\leq j\leq s)>0$ for
 any open sets $A_{i,j}$ so that  $a_i\in A_{i,j}$ for all $i$ and $j$.
When (almost) all the $T_\om$'s preserve the same absolutely continuous (with respect to $\textbf{m}$) measure $\ka$,
then our results also hold true for non-invertible $\te$'s such as (non-uniform) Young towers (see \cite{Young2} and \cite{Chaz}) with at least one periodic point and exponential tails. In fact,
we obtain also results for compositions of random maps having the form
$T_\om(x)=\boldsymbol T_{\cO(\om)}(x)$, where $\boldsymbol T_q(x)$ is a H\"older continuous function of the variable $q\in\Om^\bbN$, $\te$ is the shift map and $\cO(\om)=\{\te^n\om:\,n\geq0\}$ is the orbit  of a topologically mixing subshift of finite type or a Young tower with the properties described above.


\section{Preliminaries and main results}\label{sec2}\setcounter{equation}{0}
Our setup consists of a complete probability space 
$(\Om,\cF,P)$ together with a $P$-preserving ergodic transformation $\te:\Om\to\Om$ and
of a compact metric space $(\cX,\rho)$ together with the Borel $\sig$-algebra $\cB$.
Set $\cE=\Om\times\cX$ and let
\[
\{T_\om: \cX\to\cX,\, \om\in\Om\}
\]
be a collection of continuous bijective maps between $\cX$ to itself
so that the map $(\om,x)\to T_\om x$ is measurable with respect to $\cF\times\cB$.
Consider the 
skew product transformation $T:\cE\to\cE$ given by 
\begin{equation}\label{Skew product}
T(\om,x)=(\te\om,T_\om x).
\end{equation}
For any $\om\in\Om$ and $n\in\bbN$ consider the $n$-th step iterates $T_\om^n$ given by
\begin{equation}\label{T om n}
T_\om^n=T_{\te^{n-1}\om}\circ\cdots\circ T_{\te\om}\circ T_\om: \cX\to\cX.
\end{equation}
Then $T^n(\om,x)=(\te^n\om,T_\om^nx)$.
The main results of this paper are probabilistic limit theorems for random 
variables of the form $S_nu(\om,x)$, where $S_n=\sum_{j=0}^{n-1}u\circ T^n$, $u=u(\om,x)$
is a function satisfying certain regularity conditions and $(\om,x)$ is distributed according
to some special $T$-invariant probability measure $\boldsymbol{\mu}$. 
Our additional requirements concerning the family of maps $\{T_\om:\om\in\Om\}$
are collected in the following assumptions which are similar to \cite{MSU}.
\begin{assumption}[Topological exactness]\label{TopExAssRand}
There exist a constant $\xi>0$ and a random variable $n_\om\in\bbN$ such that 
$P$-a.s.,
\begin{equation}\label{TopExRand}
T_\om^{n_\om}(B(x,\xi))=
\cX\,\text{ for any } x\in\cX
\end{equation}
where for any  $x\in\cX$ and $r>0$, 
$\,B(x,r)$ denotes a ball in $\cX$ around $x$ with radius $r$.
\end{assumption}

\begin{assumption}[The pairing property]\label{Ass pair prop}
There exist random variables $\gam_\om>1$ and $D_\om$ such that $P$-a.s.
for any $x,x'\in\cX$ with $\rho(x,x')<\xi$ 
we can write 
\begin{equation}\label{Pair1.0}
T_\om^{-1}\{x\}=\{y_1,....,y_k\}\,\,\text{ and }\,\,T_\om^{-1}\{x'\}=\{y_1',...,y_k'\}
\end{equation}
where $\xi$ is specified in Assumption \ref{TopExAssRand},
\[
k=k_{\om,x}=|T_\om^{-1}\{x\}|\leq D_\om,
\]
$|\Gam|$ is the cardinality of a finite set $\Gam$
and
\begin{equation}\label{Pair2.0}
\rho(y_i,y_i')\leq (\gam_\om)^{-1}\rho(x,x')
\end{equation}
for any $1\leq i\leq k$.
\end{assumption}

Next, let $\al\in(0,1]$. For any measurable function $g:\cE\to\bbC$ and $\om\in\Om$ consider 
the function $g_\om:\cX\to\bbC$ given by $g_\om(x)=g(\om,x)$. Set
\begin{eqnarray*}
v_{\al,\xi}(g_\om)=\inf\{R: |g_\om(x)-g_\om(x')|\leq R\rho^\al(x,x')\,\text{ if }\,
\rho(x,x')<\xi\}\\
\text{and }\,\,\,\|g_\om\|_{\al,\xi}=\|g_\om\|_\infty+v_{\al,\xi}(g_\om)\hskip1cm
\end{eqnarray*}
where $\|\cdot\|_\infty$ is the supremum norm and $\rho^\al(x,x')=\big(\rho(x,x')\big)^\al$. 
The norms $\|g_\om\|_{\al,\xi}$ are $\cF$-measurable as a consequence
of Lemma 5.1.3 in \cite{book}. We denote by $\cH^{\al,\xi}$ the space of all functions 
$f:\cX\to\bbC$ so that $\|f\|_{\al,\xi}<\infty$.

Let $\phi,u:\cE\to\bbR$ be two measurable functions so that $P$-a.s. we have $\phi_\om,u_\om\in\cH^{\al,\xi}$.
Let $z\in\bbC$ and consider the transfer operators $\cL_z^\om,\,\om\in\Om$ which 
map functions on $\cX$ to functions on $\cX$ by
the formula
\begin{eqnarray}\label{TraOp}
\cL^\om_z g(x)=\sum_{y\in T_\om^{-1}\{x\}}e^{\phi_\om(y)+zu_\om(y)}g(y).
\end{eqnarray}
For any $n\geq1$ set 
\[
\cL_z^{\om,n}=\cL_z^{\te^{n-1}\om}\circ\cdots\circ\cL_z^{\te\om}\circ\cL_z^\om.
\]
Then 
$
\cL_z^{\om,n}g(x)=\sum_{y\in (T_\om^n)^{-1}\{x\}}e^{S_n^\om \phi(y)+zS_n^\om u(y)}g(y),
$
where $S_n^\om \psi=\sum_{j=0}^{n-1}\psi_{\te^j\om}\circ T_\om^j$  for any $\psi:\cE\to\bbC$.
Henceforth we will rely on
\begin{assumption}\label{bound ass+Tr Op meas+ass phi u}
(i) The random variables $n_\om,D_\om,\|\phi_\om\|_{\al,\xi}$ and $\|u_\om\|_{\al,\xi}$ 
are bounded and $\gam_\om-1$ is bounded from below by some positive constant.
 
(ii) The transfer operators $\cL_z^\om$ are measurable, namely the map $(\om,x)\to\cL^\om_zg_\om(x)$
is measurable, for any measurable function $g:\cE\to\bbC$.
\end{assumption}
Since $\phi$ and $u$ are measurable, $\cL_z,\,z\in\bbC$ are measurable when
for each $y\in\cX$, the map $(\om,x)\to\bbI_{T_\om^{-1}\{x\}}(y)$ is measurable with 
respect to $\cF\times\cB$, where $\bbI_C$ is the indicator function of a set $C$.
Under Assumption \ref{bound ass+Tr Op meas+ass phi u} (i) we have
$\cL_z^{\om,n}(\cH_\om^{\al,\xi})\subset\cH_{\te^n\om}^{\al,\xi}$
and the corresponding operator norm satisfies 
$\|\cL_{it}^{\om,n}\|_{\al,\xi}\leq B(1+|t|)$ where $B$ is some constant (see Lemma 5.6.1 in \cite{book} and 
Section \ref{Additional}).

\subsection{Results for invertible base maps}\label{InvRes}
We assume here that the measure preserving system 
$(\Om,\cF,P,\te)$ is invertible.
Let $\boldsymbol{\mu}=\int\boldsymbol{\mu}_\om dP(\om)$ be the $T$-invariant Gibbs measure
associated with $T_\om$ and $\phi_\om$.
Namely, $\boldsymbol{\mu}_\om$ has the form $\boldsymbol{\mu}_\om=\textbf{h}_\om(0)\boldsymbol{\nu}_\om(0)$
where $\textbf{h}_\om(0)$ and $\boldsymbol{\nu}_\om(0)$ are members of the (random) RPF triplets
of the family $\cL_0^\om,\,\om\in\Om$, as described in Section \ref{sec3}. Note that
the measure $\boldsymbol{\mu}$ coincides with the Gibbs measures
in the setup of either \cite{MSU} or \cite{Kifer-Thermo}.
Recall that under our conditions (see \cite{Kifer-1998}), the limit 
$\sig^2=\lim_{n\to\infty}\frac1n\text{var}_{\boldsymbol{\mu}_\om(0)}S_n^\om u$ 
exists $P$-a.s. and
it does not depend on $\om$. Moreover, $\sig^2>0$ if and only if 
 $\bar u_\om=u_\om-\boldsymbol{\mu}_\om(u_\om)$ does not admit a coboundary 
representation of the form
\[
\bar u_\om=q_\om-q_{\te\om}\circ T_\om
\]
for some function $q=q(\om,x)\in L^2(\cE,\cF\times\cB,\boldsymbol{\mu})$.
Consider the functions $S_nu:\cE\to\bbR$ given by 
$S_n u=\sum_{j=0}^{n-1}u\circ T^n$.
We first state

\begin{theorem}\label{CLT}
Suppose that assumption \ref{bound ass+Tr Op meas+ass phi u} is satisfied
and that $\boldsymbol{\mu}_\om(u_\om)=\int u_\om d\boldsymbol{\mu}_\om=0$, $P$-a.s. Then the CLT holds 
true, namely for any $r\in\bbR$,
\[
\lim_{n\to\infty}\boldsymbol{\mu}\{(\om,x):n^{-\frac12} S_n u(\om,x)\leq r\}=\frac1{\sqrt{2\pi}\sig}\int_{-\infty}^{r}e^{-\frac{t^2}{2\sig^2}}dt
\]
where when $\sig=0$ the above right hand side is defined to be $1$ if $r\geq 0$ and $0$ if $r<0$.
\end{theorem}
Note that Theorem \ref{CLT} follows from the quenched 
CLT in \cite{Kifer-1998} (via integration), but for readers' convenience we will prove it using 
the random complex RPF theorem. The proof will be very short and rely on
the arguments in the proof of the quenched Berry-Esseen theorem in \cite{book}. Our main interest in this 
paper is in finer limit theorems such the LLT and renewal theorem, which, in general, 
do not follow from the quenched ones (when exist), but the above CLT is needed in Theorem \ref{LLT--D.S.} below.

In order to obtain LLT's and renewal theorems we will need the following

\begin{assumption}\label{PerPointAss}
The space $\Om$ is a topological space, $\cF$  contains in the corresponding 
Borel $\sig$-algebra and $\te$ has a periodic point, namely there exist $\om_0\in\Om$ and $n_0\in\bbN$ so that $\te^{n_0}\om_0=\om_0$. Moreover, for any compact interval $J$ the maps 
$\om\to\cL_{it}^\om$ are uniformly continuous (with respect to the operator norm $\|\cdot\|_{\al,\xi}$) at the points $\te^{j}\om_0,\,0\leq j<n_0$ when $t$ ranges over $J$.
\end{assumption}
The condition about continuity of $\om\to\cL_{it}^\om$ holds true, for instance, 
when the maps $\om\to \phi_\om, u_\om\in\cH^{\al,\xi}$ are continuous at the 
points $\te^{j}\om_0,\, 0\leq j<n_0$ and $\om\to T_\om$ is either locally constant there or is continuous 
at these points with respect to an appropriate topology. 

Next, for any $n\geq1$ consider the function $V_n:\Om\to\bbR$ given by 
$V_n(\om)=\text{var}_{\boldsymbol{\mu}_\om}S_n^\om u=\boldsymbol{\mu}_\om(S_n^\om\bar u)^2$. 
We will also use the following probabilistic growth type conditions. 
\begin{assumption}\label{theta mixing assumption}
The asymptotic variance $\sig^2$ is positive. Moreover,
there exist $\be>0$ and $c_1,c_2>0$ so that for  any $n\geq1$,
\begin{equation}\label{V-mix}
P\{\om: V_n(\om)\leq c_1n\}\leq \frac{c_2}{n^\be}.
\end{equation}
Furthermore, for any $s>0$ and neighborhoods $B_{\te^{j}\om_0}$ of $\te^j\om_0$, $j=0,1,...,n_0-1$
the set $B_{\om_0,n_0,s}=\bigcap_{i=0}^{sn_0-1}\te^{-i}B_{\te^{i}\om_0}$ satisfies that for any $n\geq1$,
\begin{equation}\label{Neighb mix}
P\{\om: \sum_{j=0}^{n-1}\bbI_{B_{\om_0,n_0,s}}(\te^j\om)\leq cn\}
\leq \frac{d}{n^\be}
\end{equation}
where $c$ and $d$ are positive constants which depend only on
$s$, $\om_0$ and $n_0$ and  $\bbI_B$ is the indicator function of 
a set $B$. 

\end{assumption}
When $\cX$ is a $C^2$-compact Riemanian 
manifold, $T_\om$ are certain $C^2$ random endomorphisms 
and  $e^{-\phi_\om}$ is the corresponding Jacobian then 
$\nu_\om(0)=\textbf{m}$ (see  Theorem 2.2 in \cite{Kifer-Thermo}) , where $\textbf{m}$ is the (normalized) volume measure. 
In these circumstances, in Proposition \ref{Prop1}, we will show that condition (\ref{V-mix}) 
is satisfied when polynomial concentration inequalities of the form (\ref{Poly Vk})
hold true. In fact, the above also holds true when all $T_\om$'s are nonsingualr with respect to some measure 
$\textbf{m}$ (see Proposition \ref{Prop2}) and $\phi_\om=-\ln\big(\frac{d(T_\om)_*\textbf{m}}{d\textbf{m}}\big)$.
When $T_\om$ satisfy certain regularity conditions as function of $\om$, 
in Propositions \ref{Prop3}, we show that conditions (\ref{Poly Vk}) and (\ref{Neighb mix}) are satisfied 
when $\te$ is a (non-uniform) Young tower (see \cite{Young2} and \cite{Chaz}) with one at least one periodic point and polynomial tails. Of course, Young towers are not invertible, but in Section 
\ref{Non Inv Sec} we explain how to obtain results under Assumption \ref{theta mixing assumption}
for non-invertible maps. 
In Propositions \ref{Prop4} and \ref{Prop5} we show that the latter conditions also hold true when $\te$ is a 
Markov shift generated by a Markov chain with positive transition densities satisfying the Doeblin condition. 
We refer the readers' to Section \ref{MixSec} for  precise statements and full details.

As usual (see, for instance, \cite{GH} and  \cite{HenHerve}), in order to present the local limit theorem and the renewal theorem
we will distinguish between two cases. 
Under Assumption \ref{PerPointAss}, we call the case a non-lattice one if the 
function $S_{n_0}^{\om_0}u=\sum_{j=0}^{n-1}u_{\te^j\om_0}\circ T_{\om_0}^j$ is non-arithmetic
(aperiodic) with respect to the map $\tau=T^{n_0}_{\om_0}$ in the classical sense of \cite{HenHerve}, namely the spectral radius (with respect to the norm $\|\cdot\|_{\al,\xi}$) of the operators 
$R_{it}=\cL_{it}^{\om_0,n_0},\,t\not=0$  are strictly less than $1$.
We  call the case  a lattice one if there exists $h>0$ such that 
$P$-a.s. the function $u_\om$ takes values on 
the lattice $h\bbZ=\{hk: k\in\bbZ\}$ and the spectral radius of the operators 
$R_{it},\,t\in(-\frac{2\pi}{h},\frac{2\pi}h)\setminus\{0\}$ are strictly less than $1$. 
We refer the readers to \cite{HenHerve} for conditions equivalent to the above 
lattice and non-lattice conditions.

Next, for any $r>0$ let  $C_{r\downarrow}(\bbR)$ be the space of all continuous functions 
$g:\bbR\to\bbC$ so that $\lim_{x\to\infty}x^r g(x)=0$. 

\begin{theorem}[Local limit theorem]\label{LLT--D.S.}
Suppose that Assumptions \ref{bound ass+Tr Op meas+ass phi u}, \ref{PerPointAss} and \ref{theta mixing assumption} hold true, where in the last assumption we require that $\beta>\frac12$.
Moreover, assume that $\boldsymbol{\mu}_\om(0)(u_\om)=\int u_\om d\boldsymbol{\mu}_\om(0)=0$, $P$-a.s. 
Then for any $g\in C_{2\downarrow}(\bbR)$,
\[
\lim_{n\to\infty}\sup_{a\in\cT_h}\big|\sig\sqrt{2\pi n}\boldsymbol{\mu}\big(g (S_n u -a)\big)-\ka_h(g) e^{-\frac{a^2}{2n\sig^2}}\big|=0
\] 
where in the lattice case $\ka_h$ is the measure assigning unit mass to each member
of the lattice $\cT_h=h\bbZ$, while in the non-lattice
case we set $h=0$ and take $\ka_0$ and $\cT_0$ to be the Lebesgue measure and the real line, respectively. 
\end{theorem}

\begin{theorem}[Renewal theorem]\label{renewal-D.S.}
Suppose that Assumptions \ref{bound ass+Tr Op meas+ass phi u}, \ref{PerPointAss} and \ref{theta mixing assumption} hold true,
where in the last assumption we require that $\beta>1$.
Moreover, assume that
 $\boldsymbol{\mu}_\om(0)(u_\om)=\gam>0$ does not depend on $\om$.
Let $f:\cE\to\bbR$  be a positive function so that $f_\om\in\cH_{\al,\xi}^\om$,
$\boldsymbol{\mu}_\om(0)(f_\om)=\boldsymbol{\mu}(f)$ does not depend on $\om$ and that $\|f_\om\|_{\al,\xi}\in L^p(\Om,\cF,P)$ for some $1<p\leq\infty$ so that $\be(1-\frac1p)>1$. For any Borel measurable set $B\subset\bbR$
set  
\[
U(B)=U_{\mu,f}(B)=\sum_{n\geq1}\boldsymbol{\mu}(f\bbI_{B}(S_n u))
\]
where $\bbI_B$ is the indicator function of the set $B$.
Then in both lattice and non-lattice
cases $U$ is a Radon measure on $\bbR$ so that $\int |g|dU<\infty$ for any
$C_{4\downarrow}(\bbR)$. Moreover, for any function $g\in C_{4\downarrow}(\bbR)$, 
\begin{equation}\label{renewal equations}
\lim_{a\to-\infty}U(g_a)=0\,\,\,\text{ and }\,\,\,
\lim_{a\to\infty}U(g_a)=\frac{\boldsymbol{\mu}(f)}{\gam}\ka_h(g)
\end{equation} 
where $g_a(x)=g(x-a)$, and in the lattice case $\ka_h$ is the measure assigning unit mass to each member of the lattice $h\bbZ$, while in the non-lattice case we set $h=0$ and take $\ka_0$ to  be the Lebesgue measure.
\end{theorem}

\subsection{The non-invertible case}\label{Non Inv Sec}
Suppose that $(\Om,\cF,P,\te)$ is ergodic and not necessarily invertible.
Henceforth, we will assume  that all the maps $T_\om$ are non-singular with respect to some
probability measure $\textbf{m}$ on $\cX$ so that $\text{supp }\textbf{m}=\cX$, that  $\phi_\om=-\ln\big(\frac{d(T_\om)_*\textbf{m}}{d\textbf{m}}\big)$ 
and that $P$-a.s. the map $T_\om$ preserves a measure $\ka$ of the form $\ka=\bar h \textbf{m}$, where $\bar h$ is some 
continuous nonnegative function on $\cX$ which does not depend on $\om$. 
The latter condition means that the skew product map $T(\om,x)=(\te\om,T_\om x)$
preserves the product measure  $\boldsymbol{\mu}=P\times\ka$, whose disintegrations $\boldsymbol{\mu}_\om$ equal $\ka$.
Existence of such product measures was studied in \cite{Aimino} (see also references therein).
Note that in \cite{Aimino} the authors considered independent maps $T_\om,T_{\te\om},T_{\te^2\om},...$, but 
the $T$-invariance of $\boldsymbol{\mu}$ defined above depends only on the distribution of $T_\om$,
and not on the dependencies between $T_{\te^j\om}$'s.

\begin{theorem}\label{Non invert LimThms}
The limit 
$\sig^2=\lim_{n\to\infty}\frac1n\text{var}_{\ka}S_n^\om u$ 
exists $P$-a.s. and
it does not depend on $\om$. Moreover, $\sig^2>0$ if and only if 
 $\bar u_\om=u_\om-\ka(u_\om)$ does not admit a coboundary 
representation of the form
\[
\bar u_\om=q_\om-q_{\te\om}\circ T_\om
\]
for some function $q=q(\om,x)\in L^2(\cE,\cF\times\cB,\boldsymbol{\mu})$.
Moreover, the CLT, the local limit theorem and the renewal theorem stated in Theorems 
\ref{CLT}, \ref{LLT--D.S.} and \ref{renewal-D.S.} hold true.
\end{theorem}
We remark that the condition that $\boldsymbol{\mu}_\om(u_\om)=\ka(u_\om)$ does not depend on $\om$
is satisfied when $u_\om$  is replaced with $\frac{u_\om}{\ka(u_\om)}$. Since $\ka(u_\om)$ is H\"older
continuous in $\om$ when $u_\om$ is, all the continuity Assumptions from \ref{PerPointAss}
still hold true after this replacement. 
Theorem \ref{Non invert LimThms} is proved  by applying Theorems 
\ref{CLT}, \ref{LLT--D.S.} and \ref{renewal-D.S.} with the natural invertible extension of
$(\Om,\cF,P,\te)$, see Section \ref{Reduction}. Applying these theorems yields results
even when $T_\om$'s do not preserve the same absolutely continuous measure $\ka$, but then 
the conditions for the limit theorems to hold true (derived from the results in the extension) 
are not natural (see Remark \ref{Reduction remark}). We also note that the situation
described in Section \ref{Non-Iden} makes it possible to consider the case when 
$\phi_\om=-\ln\big(\frac{d(T_\om)_*\textbf{m}}{d\textbf{m}}\big)$  is only H\"older continuous on some
pieces of $\cX$ (e.g. when $\cX$ is the unit circle and $T_\om(x)=x^{m_\om}$ for some $m_\om\in\bbN$).

\section{Random complex RPF theorem and thermodynamic formalism constructions}\label{sec3}\setcounter{equation}{0}

Suppose that $(\Om,\cF,P,\te)$ is invertible.
In this section we  mainly collect  several results from \cite{book} which will be used
in the proofs of the results stated in Section \ref{sec2}.
Let $(\boldsymbol{\la}_\om(z),\boldsymbol{h}_\om(z),\boldsymbol{\nu}_\om(z))$ be the (random) RPF triplet
of the family $\cL_z^\om,\,\om\in\Om$, obtained in Corollary 5.4.2 in \cite{book}. Let the probability measure 
$\boldsymbol{\mu}_\om=\boldsymbol{\mu}_\om(0)$ be given by $\boldsymbol{\mu}_\om=\boldsymbol{h}_\om(0)d\boldsymbol{\nu}_\om(0)$. 
Then the measure $\boldsymbol{\mu}=\boldsymbol{\mu}(0)=\int\boldsymbol{\mu}_\om dP(\om)$ is $T$-invariant.
We refer the readers' to \cite{MSU} and \cite{Kifer-Thermo} for important properties of 
these measures (in the setups considered there).
Consider the (transfer) operator $\cA_z^\om$ given by $\cA_z^\om g=\frac{\cL_z^\om(g\boldsymbol{h}_\om(0))}{\boldsymbol{\la}_{\om}(0)\boldsymbol{h}_{\te\om}(0)}$,  namely the transfer operator generated by 
$T_\om$ and the function $\phi_\om+\ln\textbf{h}_\om(0)-\ln\textbf{h}_{\te\om}\circ T_\om-\ln\boldsymbol{\la}_\om(0)+zu_\om$.
Then the RPF theorem stated as Corollary 5.4.2 in \cite{book} holds true also 
for $\cA_z^\om$, as stated in the following

\begin{theorem}\label{RPF rand T.O. general}
There is a (bounded) neighborhood $U\subset\bbC$ of $0$ so that $P$-a.s. 
for any $z\in U$ there exists a triplet
$\la_\om(z)$, $h_\om(z)$ and $\nu_\om(z)$ consisting of a nonzero complex number 
$\la_\om(z)$, a complex function $h_\om(z)\in\cH_\om^{\al,\xi}$ and a 
complex continuous linear functional $\nu_\om(z)$ on $\cH_\om^{\al,\xi}$ such that 
\begin{eqnarray}\label{RPF deter equations-General}
\cA_z^\om h_\om(z)=\la_\om(z) h_{\te\om}(z),\,\,
(\cA_z^\om)^*\nu_{\te\om}(z)=\la_\om(z)\nu_\om(z)\text{ and }\\
\nu_\om(z) h_\om(z)=\nu_\om(z)\textbf{1}=1.\nonumber
\end{eqnarray}
For any $z\in U$ the maps $\om\to\la_\om(z)$ and $(\om,x)\to h_\om(z)(x),\,(\om,x)\in\cE$
are measurable and the family $\nu_\om(z)$ is measurable in $\om$. 
When $z=t\in\bbR$ then $\la_\om(t)>0$, the function $h_\om(t)$ is strictly
positive, $\nu_\om(t)$ is a probability measure and the equality 
$\nu_{\te\om}(t)\big(\cA_t^\om g)=\la_\om(t)\nu_{\om}(t)(g)$ holds true for any 
bounded Borel function $g:\cE_\om\to\bbC$.

Moreover, this triplet is analytic and uniformly bounded around $0$.
Namely, the maps 
\[
\la_\om(\cdot):U\to\bbC,  h_\om(\cdot):U\to \cH_\om^{\al,\xi}\,\text{ and }
\nu_\om(\cdot):U\to \big(\cH_\om^{\al,\xi}\big)^*
\]
are analytic, where $(\cH_\om^{\al,\xi})^*$ is the dual space of $\cH_\om^{\al,\xi}$, 
and for any $k\geq0$
there is a constant $C_k>0$ so that
\begin{equation}\label{UnifBound}
\max\Big(\sup_{z\in U}|\la_\om^{(k)}(z)|,\, 
\sup_{z\in U}\| h_\om^{(k)}(z)\|_{\al,\xi},\, \sup_{z\in U}
\|\nu^{(k)}_\om(z)\|_{\al,\xi}\Big)\leq C_k,\,\,P\text{-a.s.}
\end{equation}
where $g^{(k)}$ stands for the $k$-th derivative of a function on the complex 
plane which takes values in some Banach space and $\|\nu\|_{\al,\xi}$ is the 
operator norm of a linear functional $\nu:\cH_\om^{\al,\xi}\to\bbC$.

Furthermore, there exist constants $m_0$, $C$  and $c\in(0,1)$ so that $P$-a.s. for 
any $z\in U$,  $n\geq m_0$ and $q\in\cH_\om^{\al,\xi}$,
\begin{equation}\label{Exp Conv final.0}
\bigg\|\frac{\cA_z^{\om,n}q}{\la_{\om,n}(z)}
-\big(\nu_{\om}(z)q\big) h_{\te^n\om}(z)\bigg\|_{\al,\xi}\leq C\|q\|_{\al,\xi}
\cdot c^n
\end{equation}
where $\cA_z^{\om,n}=\cA_{z}^{\te^{n-1}\om}\circ\cA_{z}^{\te^{n-2}\om}\circ\cdots\circ\cA_{z}^{\om}$
and  $\la_{\om,n}(z)=\prod_{j=0}^{n-1}\la_{\te^j\om}(z)$.

\end{theorem}
Since $\cA_0^\om\textbf{1}=\textbf{1}$ we have $\la_\om(0)=1$ and $h_\om(0)\equiv\textbf{1}$. 
Remark that we can always take
\begin{equation}\label{RPF2}
\la_\om(z)=\frac{a_\om(z)\boldsymbol{\la}_\om(z)}{a_{\te\om}(z)\boldsymbol{\la}_\om(0)},\, 
h_\om(z)=\frac{a_\om(z)\textbf{h}_\om(z)}{\textbf{h}_\om(0)}\,\, \text{ and }\,\, 
\nu_\om(z)=(a_\om(z))^{-1}\textbf{h}_\om(0)\boldsymbol{\nu}_\om(z)
\end{equation}
where $(\boldsymbol{\la}_\om(z),\textbf{h}_\om(z),\boldsymbol{\nu}_\om(z))$ is the RPF triplet corresponding 
to $\cL_z^\om$ and
$a_\om(z)=\boldsymbol{\nu}_\om(z)(\textbf{h}_\om(0))$ (which is nonzero, see \cite{book}). 
 In particular, $\nu_\om(0)=\boldsymbol{\mu}_\om(0)$
and $\nu(0):=\int\nu_\om(0)dP(\om)=\boldsymbol{\mu}(0)$.

In the special case when $e^{-\phi_\om}$ is the appropriate Jacobian we have the following
\begin{proposition}\label{Prop2}
Let $\textbf{m}$ be a probability measure
on $\cX$ and suppose  $T_\om$ is non-singular with respect to $\textbf{m}$, $P$-a.s.
Let $\phi_\om=-\ln\big(\frac{d(T_\om)_*\textbf{m}}{d\textbf{m}}\big)$ and assume that 
it is a H\"older continuous function and that $\|\phi_\om\|_{\al,\xi}$ is a bounded
random variable. Then $\boldsymbol{\nu}_\om(0)=\textbf{m}$, $\boldsymbol{\la}_\om(0)=1$.

When $\phi_\om=-\ln\big(\frac{d(T_\om)_*\textbf{m}}{d\textbf{m}}\big)+\psi_\om-\psi_{\te\om}\circ T_\om+K_\om$
for some bounded function $\psi_\om(x)=\psi(\om,x)$ and a random variable $K_\om$, then 
$\nu_\om(0)$ is equivalent to $\textbf{m}$ and the Radon-Nykodim derivative is bounded from above and 
below by positive constants.
\end{proposition}
\begin{proof}
Suppose that $\phi_\om=-\ln\big(\frac{d(T_\om)_*\textbf{m}}{d\textbf{m}}\big)$. Then
$\cL_0^\om$ is the dual operator
of the (Koopman) operator $g\to g\circ T_\om$ with respect to $\textbf{m}$. In particular, 
$\textbf{m}(\cL_0^{\om,n} \textbf{1})=1$ for any $n\geq1$. 
Taking $\mu_n=\textbf{m}$ in the limit expression of $\boldsymbol{\nu}_\om(0)$ in (4.3.25) from \cite{book} (see also Theorem 5.3.1 from there ) we obtain that for any H\"older continuous function $g:\cX\to\bbR$, 
\begin{equation}\label{LimLam}
\int g d\boldsymbol{\nu}_\om(0)=\lim_{n\to\infty}
\frac{\textbf{m}(\cL_0^{\om,n}g)}{\textbf{m}(\cL_0^{\om,n}\textbf{1})}=
\lim_{n\to\infty}\frac{\textbf{m}(g)}{\textbf{m}(\textbf{1})}=\textbf{m}(g)
\end{equation}
where in the first equality we used the duality relation
discussed above. Since $\boldsymbol{\nu}_\om(0)$ and $\textbf{m}$ agree (as linear functionals) on 
$\cH^{\al,\xi}$ we obtain that the measures  $\boldsymbol{\nu}_\om(0)$ and $\textbf{m}$ are identical. 
Since $\boldsymbol{\la}_\om(0)=\boldsymbol{\nu}_{\te\om}(0)(\cL_0^\om\textbf{1})$
and $\textbf{m}(\cL_0^\om \textbf{1})=1$ we obtain from $\nu_\om(0)=\textbf{m}$ that $\boldsymbol{\la}_\om(0)=1$.

The proof in the case when
$\phi_\om=-\ln\big(\frac{d(T_\om)_*\textbf{m}}{d\textbf{m}}\big)+\psi_\om-\psi_{\te\om}\circ T_\om+K_\om$
proceeds in a similar way since then $\mu(\cL_0^{\om,n}g)=e^{\sum_{j=0}^{n-1}K_{\te^j\om}}\textbf{m}(ge^{\psi_\om-\psi_{\te^n\om}\circ T_\om^n})$.

\end{proof}
The equalities $\boldsymbol{\nu}_\om(0)=\textbf{m}$ and $\boldsymbol{\la}_\om(0)=1$ in 
Propositon \ref{Prop2} were stated in \cite{Kifer-Thermo} in the case when $\cX$ is a $C^2$-compact Riemanian manifold and $\textbf{m}$ is the (normalized) volume measure (in fact, the proof comes from \cite{Kifer-1991}).
In \cite{Kifer-Thermo} the random variable $\|\phi_\om\|_{\al,\xi}$ is only assumed to be integrable.
Under this integrability condition, the proof from \cite{book}  of existence of RPF trpilets proceeds similarly to \cite{Kifer-1991} for real $t$'s, so the condition that $\|\phi_\om\|_{\al,\xi}$ is not really necessary in
Proposition \ref{Prop2}.

\subsection{Pressure near $0$}
Suppose that the assumptions of Theorem \ref{RPF rand T.O. general} hold true.
We will gather several useful results concerning the behaviour of the logarithm of 
$\la_{\om,n}(z)$ around $0$ which were proved in Section 7.2 of \cite{book}. 
The first one is
\begin{lemma}\label{Lemma one}
There exists a neighborhood $U_1$ of $0$ so that $P$-a.s.
for any $k=1,2,...$ there is
an analytic function $\Pi_{\om,k}:U_1\to\bbC$ such that
\begin{equation}\label{New press rel}
e^{\Pi_{\om,k}(z)}=\la_{\om,k}(z)\,\text{ and }\,
|\Pi_{\om,k}(z)|\leq k(\ln2+\pi)
\end{equation}
for any $z\in U_1$.
\end{lemma}

The second is
\begin{lemma}\label{Press der general}
There exists a constant $Q_2>0$ so that 
$P$-a.s. for any $k\in\bbN$,
\begin{equation}\label{der1-Gen}
\Pi_{\om,k}'(0)=\int S_k^\om u\,d\nu_\om(0)
\,\,\,\text{ and }\,\,\,
|\Pi_{\om,k}''(0)-\emph{var}_{\nu_\om(0)}S_k^\om u|\leq Q_2.
\end{equation}
\end{lemma}

The third is
\begin{lemma}\label{Remainder Corollary 1}
Suppose that $\sig^2>0$. Then there exist constants $t_0,c_0>0$
such that $P$-a.s. for any $z$ with $|z|\leq t_0$ and $k\geq 1$,
\begin{equation}
|\Pi_{\om,k}(z)-z\Pi'_{\om, k}(0)-
\frac{z^2}2\Pi''_{\om,k}(0)|\leq c_0|z|^3 k.
\end{equation}
\end{lemma}

\subsection{Additional estimates}\label{Additional}
We gather here several estimates derived in \cite{book} which will be in 
constant use in the course of the proofs of the results stated in Section 
\ref{sec2}. We begin with 
the following random Lasota-Yorke type inequality
(see Lemma 5.6.1 in \cite{book}): there exists a constant $Q>0$ so that  
$P$-a.s. for any $n\geq1$,  $z\in\bbC$ and $g\in\cH^{\al,\xi}$,
\begin{eqnarray}\label{L.Y.-general}
\|\cL_z^{\om,n}g\|_{\al,\xi}\leq \|\cL_0^{\om,n}\textbf{1}\|_\infty
e^{|\Re(z)|\|S_n^\om u\|_\infty}\\
\times\big(v_{\al,\xi}(g)(\gam_{\om,n})^{-\al}
+(1+2Q)(1+\|z\|_1)\|g\|_\infty\big).\nonumber
\end{eqnarray}
where $\|z\|_1=|\Re(z)|+|\Im(z)|$ and $\Re(z)$ ($\Im(z)$) is the 
real (imaginary) part of $z$. 
In particular, $\cL_z^{\om,n}(\cH_\om^{\al,\xi})\subset\cH_{\te^n\om}^{\al,\xi}$
and the corresponding operator norms satisfy 
\[
\|\cL_{it}^{\om,n}\|_{\al,\xi}\leq B(1+|t|),\,\,t\in\bbR
\]
where $B$ is some constant. 
Next, similarly to proof of Corollary 5.12.3 in \cite{book}, there exists
a constant $B_0>1$ so that $P$-a.s. for any $n\geq1$, 
\begin{equation}\label{Norm comparison}
\|\cL_{z}^{\om,n}\|_{\al,\xi}B_0^{-1}\leq
 \| \cA_{z}^{\om,n}\|_{\al,\xi}\leq B_0\|\cL_{z}^{\om,n}\|_{\al,\xi}
\end{equation}
and therefore $\|\cA_{it}^{\om,n}\|_{\al,\xi}\leq B_2(1+|t|)$ where $B_2=BB_0$. 

Finally, since $\la_{\om,n}(z)=\nu_\om(z)(\cA_z^{\om,n}\textbf{1})$, applying
(\ref{UnifBound}) and (\ref{Norm comparison}) yields that there exists a constant $A$ so that $P$-a.s.,
\begin{equation}\label{lam basic bound}
|\la_{\om,n}(it)|\leq A\|\cA_{it}^{\om,n}\|_{\al,\xi}\leq AB_2
\end{equation}
for any $n\geq1$ and real $t$ so that $it\in U$.
Note that in the integral operator case it is trivial that 
$\|R_{it}^{\om,n}\|_{\infty}\leq 1$, so there is no need in using additional results in order to obtain appropriate upper bounds on the latter norms.

\section{Annealed limit theorems}\label{sec4}\setcounter{equation}{0}

\subsection{Reduction to the invertible case}\label{Reduction}
Let $(\Om,\cF,P,\te)$ be a measure preserving system 
and let $T_\om,\phi_\om$ and $u_\om$ be as described at the beginning of Section
\ref{sec2}. Recall the following definition. We say that $(\hat\Om,\hat\cF,\hat P,\hat\te)$
is the natural (invertible) extension of $(\Om,\cF,P,\te)$ if $\hat\Om\subset\Om_0^\bbZ$ is the space of all two sided sequences $\hat\zeta=(...\zeta_{-1},\zeta_0,\zeta_1...)$ so that $\te\zeta_i=\zeta_{i+1}$ for any $i$,\, $\hat\te$ is the shift map defined by $(\hat\te\zeta)_i=\zeta_{i+1}$,\, $\hat\cF$ is the $\te$-algebra induced on $\Om$ by the product $\sig$-algebra $\cF^\bbZ$ and $\hat P$ is the probability measure defined by 
$\hat P\{\hat\zeta: \zeta_i\in A_i;\,|i|\leq M\}=P_0(\bigcap_{i=0}^{2M}\te^{-i}A_{i-M})$.
When $\Om$ is a metric space and $\cF$ contains its Borel $\sig$-algebra then $\hat\Om$ is a metric 
space and $\hat\cF$ contains the appropriate Borel $\sig$-algebra.
Set $\cT_{\hat\zeta}=T_{\zeta_0}$ and consider the skew product map $\cT$ given by 
$\cT(\hat\zeta,x)=(\hat\te\hat\zeta,T_{\hat\zeta}x)$. Let $\pi_0:\hat\Om\to\Om$ be the projection
on the $0$-th coordinate given by $\pi_0\hat\zeta=\zeta_0$. 
Then $\pi_0$ is a factor map.
Set $\pi(\hat\zeta,x)=(\pi_0\hat\zeta,x)=(\zeta_0,x)$
and let $\boldsymbol{\mu}=P\times\ka$ be the measure described in Section \ref{Non Inv Sec}. Then $\boldsymbol{\hat\mu}:=\pi_*\boldsymbol{\mu}$
is the Gibbs measure described before Theorem \ref{CLT}, i.e. the measure whose disintegrations
are given by $\boldsymbol{\mu}_{\hat\zeta}=\boldsymbol{h}_{\hat\zeta}(0)\boldsymbol{\nu}_{\hat\zeta}(0)$. Indeed, in our circumstances Proposition
\ref{Prop2} shows that $\boldsymbol{\nu}_{\hat\zeta}(0)=\textbf{m}$ and that $\boldsymbol{\la}_{\hat\zeta}(0)=1$. Therefore, it will be sufficient to show that $\,\boldsymbol{h}_{\hat\zeta}=\bar h$, $\hat P$-a.s.. For this purpose, for any continuous function $g$ on $\cX$ write
\[
\ka(g)=\textbf{m}(\bar h g)=\textbf{m}(\bar h \cdot g\circ T_{\zeta_0})=\boldsymbol{\nu}_{\hat\zeta}(0)(\bar h \cdot g\circ\cT_{\hat\zeta})
=\boldsymbol{\nu}_{\hat\te\hat\zeta}(0)(\cL_0^{\hat\zeta}\bar h \cdot g)=
\textbf{m}(\cL_0^{\hat\zeta}\bar h \cdot g)
\] 
where in the second equality we used the $T$-invariance of $\boldsymbol{\mu}=P\times\ka$.
Since these equalities hold true for any continuous $g$ and $\bar h$ is continuous, we derive
that $\cL_0^{\hat\zeta}\bar h=\bar h$. Now, as in the proof of Propostion 3.19 in \cite{MSU}, 
we derive from  (\ref{Exp Conv final.0}) that the left hand side of (\ref{Exp Conv final.0}) converges 
to $0$ as $n\to\infty$ for any continuous function $q:\cX\to\bbC$.
Since $\bar h$ is a density function we have $\textbf{m}(\bar h)=1$,
and therefore, taking $q=\bar h$ yields that the sequence $\boldsymbol{h}_{\hat\te^n\hat\zeta}(0)$ converges to $\bar h$.
We claim next that 
$\boldsymbol{h}_{\hat\zeta}=\bar h$, $\hat P$-a.s. Indeed, for any $x\in\cX$ we have 
\[
0=\lim_{n\to\infty}\frac1n\sum_{j=0}^{n-1}|\boldsymbol{h}_{\hat\te^j\hat\zeta}(0)(x)-\bar h(x)|,\,\,\hat P-a.s.
\]
implying that $\int |\boldsymbol{h}_{\zeta}(0)(x)-\bar h(x)|d\hat P(\hat\zeta)=0$, or, equivalently, that 
$\boldsymbol{h}_{\zeta}(0)(x)=\bar h(x)=0$, $\hat P$-a.s. By compactness of $\cX$ and continuity of 
both $\boldsymbol{h}_{\zeta}(0)$ and $\bar h$ we obtain the desired equality.

Set $U_{\hat\zeta}=u_{\zeta_0}$, and $\Phi_{\hat\zeta}=\phi_{\zeta_0}$. Then 
$\boldsymbol{\hat\mu}_{\hat\zeta}(U_{\hat\zeta})=\ka(u_{\zeta_0})=\boldsymbol{\mu}_{\zeta_0}(u_{\zeta_0})$
and the distribution of the processes $\{(U\circ\cT^n)(\hat\zeta,x):n\geq0\}$ and 
$\{(u\circ T^n)(\om,y):n\geq0\}$ are the same when $(\hat\zeta,x)$ is distributed according to 
$\boldsymbol{\hat\mu}$ and $(\om,y)$ is distributed according to $\boldsymbol{\mu}$. Finally,
the periodic points of $\te$ are exactly the points of the form $\hat\zeta=(...a,a,a,...)$, where 
$a=(\om_0,\te\om_0,...,\te^{n_0-1}\om_0)$ for some periodic point $\om_0$ of $\te$ whose 
period is $n_0$. Therefore,
all the conditions in Theorems \ref{CLT}, \ref{LLT--D.S.} and \ref{renewal-D.S.} hold true
with $\cT_{\hat\zeta},\Phi_{\hat\zeta}$ and $U_{\hat\zeta}$ when they hold with $T_{\om},\phi_{\om}$
and $u_\om$.

\begin{remark}\label{Reduction remark}
Applying the results stated in Section \ref{InvRes},
we obtain limit theorems in the non-invertible case also when $T_\om$'s do not necessarily preserve 
the absolutely continuous measure. Indeed, let $\boldsymbol{\hat\mu}$ be the 
Gibbs measure (in the extension) and set $\boldsymbol{\mu}=\pi_*\boldsymbol{\hat\mu}$. Then $\boldsymbol{\mu}$
is $T$-invariant and the distribution of the processes $\{(U\circ\cT^n)(\hat\zeta,x):n\geq0\}$ and 
$\{(u\circ T^n)(\om,y):n\geq0\}$ are the same when $(\hat\zeta,x)$ is distributed according to 
$\boldsymbol{\hat\mu}$ and $(\om,y)$ is distributed according to $\boldsymbol{\mu}$. The problem here
is that the condition that $\boldsymbol{\hat\mu}_{\hat\zeta}(U_{\hat\zeta})$ does not depend on $\hat\zeta$
can not be easily expressed in terms of the original non-invertible system $(\Om,\cF,P,\te)$, as explained in
the following.
The limit expressions (4.3.25) and (4.3.28) in Chapter 4 of \cite{book} show that $\boldsymbol{\la}_{\hat\zeta}(z)$
and $\boldsymbol{\nu}_{\hat\zeta}(z)$ depend only on $\zeta_0$, but in general, 
the function $\boldsymbol{h}_{\hat\zeta}(z)$ does not depend only on $\zeta_0$ since (see (\ref{Exp Conv final.0})),
\[
\boldsymbol{h}_{\hat\zeta}(z)=\lim_{n\to\infty}\frac{\hat\cL_{z}^{\hat\te^{-n}\hat\zeta,n}\textbf{1}}{\boldsymbol{\la}_{\hat\te^{-n}\hat\zeta,n}(z)}.
\]
Therefore, the Gibbs measure $\boldsymbol{\mu}_{\hat\zeta}(0)$ does not depend only on $\hat\zeta_0$.
Since $\hat\te$ is ergodic, the condition that $\boldsymbol{\mu}_{\hat\zeta}(0)(U_{\hat\zeta})$
does not depend on $\hat\zeta$ is equivalent to existence of the limit 
$\lim_{n\to\infty}\boldsymbol{\mu}_{\hat\te^n\hat\zeta}(0)(U_{\hat\te^n\hat\zeta})$, $\hat P$-a.s.,  which together with (\ref{Exp Conv final.0}), is equivalent to existence of the limits of 
\[
\big(\boldsymbol{\la}_{\hat\zeta,n}(0)\big)^{-1}\boldsymbol{\nu}_{\hat\te^n\hat\zeta}(0)(U_{\hat\te^n\hat\zeta}\hat\cL_0^{\hat\zeta,n}\textbf{1})=
\boldsymbol{\nu}_{\hat\zeta}(0)(U_{\hat\te^n\hat\zeta}\circ T_{\hat\zeta}^n)
\]
as $n\to\infty$ (note that the latter expressions depend only on $\zeta_0$). By proposition \ref{Prop2}, 
when $\phi_\om=-\ln\big(\frac{d(T_\om)_*\textbf{m}}{d\textbf{m}}\big)$ then
$\boldsymbol{\nu}_{\hat\zeta}(0)=\textbf{m}$, and so in this case the condition that
$\boldsymbol{\mu}_{\hat\zeta}(0)$ does not depend on $\hat\zeta$ is equivalent to 
convergence of $\textbf{m}(u_{\te^n\om}\circ T_\om^n)$ to $\boldsymbol{\mu}(u)=\int u d\boldsymbol{\mu}$.
\end{remark}

\subsection{Characteristic functions and Markov chains}
We begin with the following. 
Let $\xi_{n}^{\te^{-n}\om},\,n\geq0$ be the Markov chain 
generated by the initial distribution $\boldsymbol{\mu}_\om=\nu_\om(0)$
on $\cX$ and the $n$-th step transition operators 
$\cA_0^{\te^{-n}\om,n}$, where $\cA_z^{\om,j}=\cA_{z}^{\te^{j-1}\om}\circ\cA_{z}^{\te^{j-2}\om}\circ\cdots\circ\cA_{z}^{\om}$ for any complex $z$ and natural $j$. 
Set $S_n=\sum_{j=0}^{n-1}u\circ T^n$,
\[
\cS_n^\om=\sum_{j=0}^{n-1}u_{\te^{-j}\om}(\xi_{j}^{\te^{-j}\om})
\]
and let $\psi_{\om,n}(t)$ be the characteristic function of $\cS_n^\om$.
When $x$ is distributed according to $\boldsymbol{\mu}_\om:=\boldsymbol{\mu}_\om(0)$ we have
\begin{equation}\label{DistEq}
(x,T_\om x,T^2_\om x,...,T^{n}_\om x)\overset{d}{=}(\xi_n^{\te^\om},\xi_{n-1}
^{\te^{n}\om},...,\xi_{0}^{\te^n\om})
\end{equation}
where $\overset{d}{=}$ stands for equality in distribution. 
Therefore,
\begin{equation}\label{Basic Char Fun rel}
\bbE e^{itS_n}=\int \psi_{\om,n}(t)dP(\om)=\int \psi_{\te^n\om,n}(t)dP(\om)=\int\boldsymbol{\mu}_{\te^n\om}(\cA_{it}^{\om,n}\textbf{1})dP(\om).
\end{equation}

\subsection{Proof of the CLT}
We assume here that $\nu_\om(0)(u_\om)=0$. Set 
$\hat S_n=n^{-\frac12}S_n$.
By the Levi continuity theorem, it suffice to show that there exists $r>0$ so that 
for any $t\in[-r,r]$,
\begin{equation}\label{Levi contin}
\lim_{n\to\infty}\bbE e^{it\hat S_n}=e^{-\frac12\sig^2t^2}.
\end{equation}
In order to prove the latter equality, we first write
\[
\bbE e^{it\hat S_n}=\int\boldsymbol{\mu}_{\te^n\om}(\cA^{\om,n}_{it_n}\textbf{1})dP(\om)
\]
where $t_n=n^{-\frac12}t$. Set 
\[
\varphi_{\om,n}(z)=\int\frac{\cA^{\om,n}_{z}\textbf{1}}{\la_{\om,n}(z)}d\boldsymbol{\mu}_{\te^n\om}.
\]
Then we can write
\[
\bbE e^{it\hat S_n}=\int e^{\Pi_{\om,n}(it_n)}\varphi_{\om,n}(it_n)dP(\om)
\]
where $\Pi_{\om,n}$ is defined in Lemma \ref{Lemma one}.
As in the proof of Theorem 7.1.1 in  \cite{book}, 
there exist constants $B,b>0$ such that for any $n\geq 1$ and $z$ so that $|z|\leq b$,
\begin{equation}\label{phi est}
|\varphi_{\om,n}(z)-\varphi_{\om,n}(0)|=|\varphi_{\om,n}(z)-1|\leq B|z|^2.
\end{equation}
Let $\ve\in(0,1)$ and $N_\ve$ be so that for any $n\geq N_\ve$ we have $P(\Gam_{n,\ve})\geq1-\ve$, 
where
\[
\Gam_{n,\ve}=\{\om: |n^{-1}\text{var}_{\nu_\om(0)}S_n^\om u-\sig^2|\leq\ve\}.
\]
Then for any $n\geq N_\ve$ and $t\in[-b,b]$,
\[
|\bbE e^{it\hat S_n}-e^{-\frac12\sig^2t^2}|\leq 
2C\ve+\int_{\Gam_{n,\ve}}|e^{\Pi_{\om,n}(it_n)}\varphi_{\om,n}(it_n)-e^{-\frac12\sig^2t^2}|dP(\om)
\]
where $C>0$ is some constant, and we used (\ref{lam basic bound}) and that
 $e^{\Re\bar\Pi_{\om,n}(it)}=|\la_{\om,n}(it)|$. Next, by lemma 
\ref{Press der general}, Lemma \ref{Remainder Corollary 1},\, 
(\ref{lam basic bound}) and (\ref{phi est}) there exists a constant $D>0$ so that 
for any $t\in[-b,b]$, $\ve>0$, $n\geq N_\ve$ and $\om\in\Gam_{n,\ve}$, 
\begin{eqnarray*}
|e^{\Pi_{\om,n}(it_n)}\varphi_{\om,n}(it_n)-e^{-\frac12\sig^2t^2}|\leq
e^{\Re\bar\Pi_{\om,n}(it_n)}
|\varphi_{\om,n}(it_n)-1|\\+
|e^{\Pi_{\om,n}(it_n)}-e^{-\frac12\sig^2t^2}|\leq D(n^{-1}+n^{-\frac12}+e^\ve-1).
\end{eqnarray*}
Equality (\ref{Levi contin}) follows now by taking $n\to\infty$
and then $\ve\to0$.

\subsection{Norms estimates for small t's}
For $P$ a.a. $\om$ and $n\in\bbN$, write
\[
|\la_{\om,n}(it)|=e^{\Re\Pi_{\om,n}(it)}.
\] 
By Lemmas \ref{Press der general} and \ref{Remainder Corollary 1}, we derive that for any $c>0$ there
exist positive constants $\del_0$ and $c_1$ so that for any $t\in[-\del_0,\del_0]$ and 
$n\geq1$ such that $\text{var}_{\nu_\om(0)}S_n^\om u\geq cn$,  
\begin{equation}
|\la_{\om,n}(it)|\leq Ae^{-c_1nt^2}. 
\end{equation}
Therefore, under Assumption \ref{theta mixing assumption}, there exist positive
constants $d_1,d_2$ and $\del_0$ and 
sets $\Gam_n,\,n\in\bbN$ so that 
\begin{equation}\label{First Gamma props1}
1-P(\Gam_n)\leq\frac{d_1}{n^\be}
\end{equation}
and for any $\om\in\Gam_n$ and $t\in[-\del_0,\del_0]$, 
\begin{equation}\label{First Gamma props2}
 |\la_{\om,n}(it)|\leq Ae^{-d_2nt^2}.
\end{equation}
Note that by (\ref{UnifBound}) and (\ref{Exp Conv final.0}), there exists a constant 
$A'$ so that 
\[
\|\cA_{it}^{\om,n}\|_{\al,\xi}\leq A'|\la_{\om,n}(it)|
\]
and so we obtain on $\Gam_n$ appropriate estimates of $\|\cA_{it}^{\om,n}\|_{\al,\xi}$ and
$\|\cL_{it}^{\om,n}\|_{\al,\xi}$, as well.

\begin{remark}
Suppose that $\sig^2>0$ and 
set $V_{n}(\om)=\text{var}_{\nu_\om(0)}S_n^\om u$.
In Chapter 7 of \cite{book} we proved that, $P$-a.s., the converges rate of of 
law of $(V_n(\om))^{-\frac12}S_n^\om \bar u$ towards the 
standard normal law is optimal. When Assumption \ref{theta mixing assumption} hold true with 
$\beta=\frac12$, then (\ref{First Gamma props2})
and the arguments in \cite{book} yield the following result:
there exists a constant $c>0$ so that for any $n\geq1$,
\begin{equation}\label{B.E}
\sup_{r\in\bbR}\Big|\boldsymbol{\mu}\{(\om,x):\,S_n^\om\bar u_\om(x)\leq r\sqrt{V_n(\om)}\}-\frac1{\sqrt{2\pi}}\int_{-\infty}^re^{-\frac12t^2}dt\Big|\leq cn^{-\frac12}.
\end{equation}
This is not the Berry-Essen theorem for the self normalized sums 
$\big(\text{var}_{\boldsymbol{\mu}}S_n \bar u\big)^{-\frac12}S_n\bar u$.
Obtaining estimates on the left hand side of (\ref{B.E}) with
$\text{var}_{\boldsymbol{\mu}}S_n \bar u=\bbE_PV_n(\om)$ in place of
$V_n(\om)$ requires concentration inequalities of the form
\[
P\big\{\om:\,|V_n(\om)-\bbE_P V_n(\om)|\geq 
c_1n^{\del_1}\big\}\leq c_2n^{-\del_2}
\]
for some $c_1,c_2,\del_1,\del_2>0$. When $\del_1=\del_2\geq\frac12$ this would yield  the rate $n^{-\frac12}$, while in general we would get a (possibly) smaller negative power of $n$. The problem here is that such inequalities do not seem to hold true under general conditions,
as demonstrated in the following. Let $\te$ be a Young tower.
When  $\boldsymbol{\mu}_\om=\mu$ and 
$\boldsymbol{\mu}_\om(u_\om)=\gam$ do not depend on $\om$, set 
$\hat V_n(\om_0,...,\om_{n-1})=\text{var}_{\mu}
(\sum_{j=0}^{n-1}u_{\om_0}\circ T_{\om_j}\circ T_{\om_{j-1}}\circ\cdots\circ T_{\om_0})$. The function $\hat V_n$ is H\"older continuous when the family $\{T_\om(\cdot):\,\om\in\Om\}$ is uniformly H\"older continuous, and $u_\om$ and $T_\om$ are H\"older continuous functions of the variable $\om$. The  H\"older constant of $\hat V_n$ at direction $\om_0$ grows exponentially fast in $n$, which makes it impossible to apply effectively the concentration inequalities of the form used in \cite{Chaz} (see Section \ref{MixSec}), or any other reasonable general type of concentration 
inequality.
\end{remark}

\subsection{Norms estimates for large t's}
We will prove here the following
\begin{lemma}\label{Del n Lemmma}
Suppose that (\ref{Neighb mix}) from Assumption \ref{theta mixing assumption} holds true.
In the non-lattice case, let $J\subset\bbR$ be a compact set not containing the origin. In the lattice case, let $J$ be a compact subset of 
$(-\frac{2\pi}h,\frac{2\pi}h)\setminus\{0\}$. Then
there exist sets $\Del_n=\Del_n(J),\,n\geq 1$ and positive constants $d=d(J)$ 
and $u=u(J)$ so that for any $n\geq1$, 
\begin{equation}\label{DEL PROB}
1-P(\Del_n)\leq\frac{d}{n^\beta}
\end{equation}
and for any $\om\in\Del_n$,
\begin{equation}\label{Norm big t's}
\sup_{t\in J}\|\cA_{it}^{\om,n}\|_{\al,\xi}\leq 4B_0\cdot 2^{-un}
\end{equation}
where $B_0$ comes from (\ref{Norm comparison}). 
\end{lemma}
\begin{proof}
First, by (\ref{L.Y.-general})
there exists a constant $B>0$ so that $P$-a.s.
for any real $t$,
\begin{equation}\label{basic nor, bound, large t's}
\|\cL_{it}^{\om,n}\|_{\al,\xi}\leq B(1+|t|):=B_t.
\end{equation}
Consider the transfer operators $R_{it},\,t\in\bbR$ generated by the function $S_{n_0}^{\om_0}u$ and the map 
$\tau=T_{\om_0}^{n_0}$, namely
\[
R_{it}=\cL_{it}^{\om_0,n_0}.
\]
Observe that for each $t$ the spectral radius of $R_{it}$ does not exceed $1$ since the 
norms $\|\cL_{it}^{\om,n}\|_{\al,\xi}$ are bounded in $\om$ and $n$.
In the non-lattice case, let  $J\subset\bbR$ be 
a compact set not containing $0$, while in the lattice case 
let $J$ be a compact subset of $(-\frac {2\pi}h,\frac{2\pi}h)\setminus\{0\}$. 
Then, in both cases, there there exist constants $A>0$
and $a\in(0,1)$, which may depend on $J$,
so that for any $n\in\bbN$ and $t\in J$,  
\begin{equation}\label{R norms apprx1}
\|R_{it}^n\|_{\al,\xi}\leq Aa^n.
\end{equation}
The above follows from classical spectral analysis type results together with our assumptions 
about the $R_{it}$'s, and we refer the readers' to 
\cite{HenHerve} for the details. Let $J$ be a compact set as described above and set $B_J=\sup_{t\in J}B_t$.
Let $s=s(J)$ be sufficiently large so that $Aa^s<\frac{1}{4B_J}$, 
where $A$ and $a$ satisfy  (\ref{R norms apprx1}). Since Assumption \ref{PerPointAss}
holds true there exist neighborhoods $B_{\te^j\om_0},\,0\leq j<n_0$ of the 
points $\te^j\om_0,\,0\leq j<n_0$, respectively, so that 
\[
\sup_{t\in J}\|\cL_{it}^{\om,n_0s}-R_{it}^s\|_{\al,\xi}\leq\frac1{4B_J}
\]
for any 
$
\om\in B_{\om_0,n_0,s}:=\bigcap_{i=0}^{sn_0-1}\te^{-i}B_{\te^{i}\om_0},
$
where we also used that $R_{it}^s=\cL_{it}^{\om_0,sn_0}$.
It follows that for any $\om\in B_{\om_0,n_0,s}$,
\begin{equation}
\sup_{t\in J}\|\cL_{it}^{\om,n_0s}\|_{\al,\xi}\leq\frac1{2B_J}.
\end{equation}
Set 
\begin{equation}\label{Del n def.}
\Del_n=\{\om\in\Om: \sum_{j=0}^{n-1}\bbI_{B_{\om_0,n_0,s}}(\te^j\om)>cn\}
\end{equation}
where $c$ comes from (\ref{theta mixing assumption}).
Then 
\begin{equation}\label{DEL PROB.0}
1-P(\Del_n)\leq\frac{d}{n^\beta}.
\end{equation}
We also set 
\[
\textbf{B}_{\om_0,n_0,s}=\{(u_i)\in\Om^{sn_0}:\,u_i\in B_{\te^i\om_0}\,\,\,\forall\, 1\leq i\leq sn_0\}.
\]
Then $\om'\in B_{\om_0,n_0,s}$ if and only if the "word" $(\om',\te\om',...,\te^{sn_0-1}\om')$
belongs to $\textbf{B}_{\om_0,n_0,s}$. 
When $\om\in\Del_n$ the word $(\om,\te\om,...,\te^{n-1}\om)$ contains at least $[\frac{nc}{n_0s}]-1$
disjoint sub-words 
\[
\om(q,sn_0)=(\te^{q+1}\om,...,\te^{q+sn_0}\om)
\]
which are contained in  $\textbf{B}_{\om_0,n_0,s}$. Namely, there exist indexes $q_1,q_2,...,q_L,\,L\geq [\frac{nc}{n_0s}]-1$ so that 
each $\om(q_i,sn_0)$ is a member of $\textbf{B}_{\om_0,n_0,s}$ and $q_i+sn_0<q_{i+1}$ for all $i=1,2,...,L-1$.
As a consequence, on $\Del_n$ we can write
\[
\cL^{\om,n}_{it}=\cB_{t,1}^\om\circ \cC_{t,1}^\om\circ\cB_{t,2}^\om\circ
\cC_{t,j}^\om\circ\cdots\circ\cC_{t,L}^\om\circ\cB_{t,L+1}^\om
\]
where $L=L_{\om,n}\geq[\frac{cn}{n_0s}]-1$, 
each $\cC_{t,j}^{\om}$ satisfies 
\[
\sup_{t\in J}\|\cC_{t,j}^\om\|_{\al,\xi}\leq\frac1{2B_J}
\] 
while each $\cB_{t,j}^\om$ satisfies
\[
\sup_{t\in J}\|\cB_{t,j}^\om\|_{\al,\xi}\leq B_J.
\]
We conclude from the above estimates that on $\Del_n$ we have
\begin{equation}\label{Norm big t's.1}
\sup_{t\in J}\|\cA_{it}^{\om,n}\|_{\al,\xi}\leq B_0
\sup_{t\in J}\|\cL_{it}^{\om,n}\|_{\al,\xi}\leq 4B_0\cdot 2^{-un}
\end{equation}
where $u=\frac{c}{n_0s}$, and we also used (\ref{Norm comparison}).
\end{proof}

\subsection{Proof of the LLT}
By Theorem \ref{CLT}  and exactly as in the proof of Theorem 2.3 in \cite{book}, it is sufficient to show that Assumptions 2.2.1 and 2.2.2 from \cite{book} hold true. These assumptions are verified using the inequality
\begin{eqnarray*}
|\bbE e^{it S_n}|=\big|\int\mu_{\om}(\cA_{it}^{\te^{-n}\om,n}\textbf{1})dP(\om)\big|=
\big|\int\mu_{\te^n\om}(\cA_{it}^{\om,n}\textbf{1})dP(\om)\big|\\
\leq 1-P(\Gam)
+\Big\|\bbI_\Gam(\om)\|\cA_{it}^{\om,n}\|_{\al,\xi}\Big\|_{L^\infty(\Om,\cF,P)}
\end{eqnarray*}
applied with either $\Gam=\Gam_n$ defined before (\ref{First Gamma props1})
or with $\Gam=\Del_n$ from Lemma \ref{Del n Lemmma}, taking into account that  $\beta>\frac12$.


\subsection{Proof of the renewal theorem}\label{Renewal Section}
In this section we will always assume that the conditions of Theorem \ref{renewal-D.S.} hold
true. In the course of the proof, when it is more convenient, will use the notation $\bbE_P$
to denote integration over $\Om$ with respect to $P$.
Let $\mathscr H_1$ be the space of all continuous and integrable functions $g:\bbR\to\bbC$ whose Fourier 
transforms are continuously differentiable with compact support.   
Then by Lemma IV.5 in \cite{HenHerve},
it is sufficient in (\ref{renewal equations}) from Theorem \ref{renewal-D.S.}
to consider only functions $g\in\mathscr H_1$ which are dominated by some positive element of $\mathscr H_1$. Note
that such functions satisfy the Fourier inversion formula.
For any $\rho\in(0,1)$ set 
\begin{eqnarray*}
U_\rho(B)=\sum_{n\geq1}\rho^{n-1}\int\mu_{\om}
\big(f_{\te^{-n}\om}(\xi_n^{\te^{-n}\om})\bbI_B(S_n^\om)\big)dP(\om)\\=
\sum_{n\geq1}\rho^{n-1}\int\bbE_{\mu_{\te^n\om}}[f_{\om}(\xi_n^{\om})\bbI_B(S_n^{\te^n\om})]dP(\om).
\end{eqnarray*}
Since $\bbE_P\|f_\om\|_\infty<\infty$, it is clear that $U_\rho$ is a finite measure 
for each $\rho\in(0,1)$.
Let $g$  a function of the described above form and suppose that its Fourier transform 
$\hat g$ vanishes outside the compact interval $[-b,b]$.
 In the non-lattice case set
\[
V_\rho (g)=\frac{1}{2\pi}\int_{-\del_0}^{\del_0}
\hat g(t)\bbE_P\big[\nu_\om(it)(f_\om)f_\rho^\om(t)\big]dt
\]
where
\[
f_\rho^\om(t)=\sum_{n\geq1}\rho^{n-1}\la_{\om,n}(it)\mu_{\te^{n}\om}(h_{\te^{n}\om}(it))
\] 
which by  (\ref{UnifBound}) and (\ref{lam basic bound}) converges uniformly over 
$\om$ and $t\in[-\del_0,\del_0]$.
In the lattice case we set
\[
V_\rho(g)=\frac{1}{2\pi}\int_{-\del_0}^{\del_0}r(t)\bbE_P\big[\nu_\om(it)(f_\om)f_\rho^\om(t)\big]dt
\]
where $r(t)=\sum_{k\in\bbZ}\hat g(t+\frac{2\pi k}{h})$. 
Note that by the so-called Poisson summation formula (see Ch. 10 in \cite{Brei}) we have
$r(0)=\sum_{k}\hat g(\frac{2\pi k}h)=\int gd\ka_h=\ka_h(g)$.


\begin{lemma}\label{VII.3}
In both the lattice and non-lattice cases,
there exist integrable functions $R_1$ and $R_2$ on $\bbR$ so that
\[
\lim_{\rho\to1}\big(U_\rho(g)-V_\rho(g)\big)=\int_{\{t: |t|>\del_0\}}e(t)R_1(t)dt+
\int_{-\del_0}^{\del_0}e(t)R_2(t)dt
\]
where in the non-lattice case $e(t)=\hat g(t)$ while in the lattice case 
$e(t)=r(t)$. 
\end{lemma}

\begin{proof}
For any $0<\rho<1$ set 
\[
L_\rho(t)=\sum_{n\geq 1}\rho^{n-1}\bbE_P\big[\mu_{\te^n\om}(\cA_{it}^{\om,n}f_{\om})\big].
\]
Since $\|\cA_{it}^{\om,n}\|_\infty
\leq\|\cA_{0}^{\om,n}\textbf{1}\|_\infty=1$ and $\|f_{\om}\|_\infty$ is integrable
it follows that
the series $L_\rho^\om(\cdot)$ converges uniformly on $\bbR$.
Consider first the non-lattice case. 
Then by the inversion formula of the Fourier transform (applied with the 
function $g$) and the Fubini theorem, for any
$n\geq1$,
\begin{equation}\label{I.0}
\int\mu_{\te^n\om}
\big(f_{\om}(\xi_n^{\om})g(S_n^{\te^n\om})\big)dP(\om)=
\frac1{2\pi}\int_{-b}^b\hat g(t)\bbE_P\big[\mu_{\te^n\om}(\cA_{it}^{\om,n}f_{\om})\big]dt
\end{equation}
and therefore,
\begin{equation}\label{I}
U_\rho(g)=\frac1{2\pi}\int_{-b}^b\hat g(t)L_\rho(t)dt.
\end{equation}
Next, by  (\ref{Exp Conv final.0}) there exists a constant $C>0$ so that
for any $t\in[-\del_0,\del_0]$,
\begin{equation}\label{R1 bound}
\sum_{n\geq1}\bbE_P\|\cA_{it}^{\om,n}f_{\om}-
\la_{\om,n}(it)\nu_{\om}(it)(f_\om)h_{\te^n\om}(it)\|_{\al,\xi}\leq C
\end{equation}
where we used that $\|f_\om\|_{\al,\xi}$ is $P$-integrable and that 
$|\la_{\om,n}(it)|$ is bounded in $\om,n$ and $t\in[-\del_0,\del_0]$
(see (\ref{lam basic bound})). In fact, (\ref{Exp Conv final.0}) also implies that 
the series on the left hand side of (\ref{R1 bound}) converges uniformly over $t\in[-\del_0,\del_0]$.
As a consequence, we can write
\[
2\pi\big(U_\rho(g)-V_\rho(g)\big)=\int_{\del_0<|t|\leq b}\hat g(t)L_\rho(t)dt+
\int_{-\del_0}^{\del_0}\hat g(t)R_{2,\rho}(t)dt
\]
where 
\[
R_{2,\rho}(t)=\sum_{n\geq1}\rho^{n-1}\bbE_P\mu_{\te^n\om}\big(\cA_{it}^{\om,n}f_{\om}-
\la_{\om,n}(it)\nu_{\om}(it)(f_\om)h_{\te^n\om}(it)\big)
\]
which converges as $\rho\to1$ to $R_{2,1}(t)$ uniformly over $t\in[-\del_0,\del_0]$. It is clear
by (\ref{R1 bound}) that
$\sup_{t\in[-\del_0,\del_0]}|R_{2,1}(t)|\leq C$.
We claim next that $L_\rho(t)$ converges uniformly on
$[-b,b]\setminus[-\del_0,\del_0]$ towards a limit $R_1(t)$ which is a bounded
function of $t$ on this domain. This will complete the proof of the lemma in the non-lattice case. Indeed, set $J=[-b,b]\setminus[-\del_0,\del_0]$ and let $\Del_n=\Del_n(J)$ be the sets from Lemma \ref{Del n Lemmma}. Then by (\ref{DEL PROB}) and (\ref{Norm big t's})
there exist constants $d=d(J)>0$ and $u=u(J)>0$ so that 
$P(\Del_n)\geq 1-dn^{-\be}$ and for any $\om\in\Del_n$,
\[
\sup_{t\in J}\|\cA_{it}^{\om,n}\|_{\al,\xi}\leq 4B_0\cdot 2^{-un}.
\]
Therefore, with $\|\cdot\|=\|\cdot\|_{\al,\xi}$, for any $t\in J$ and $m\geq1$ we have
\begin{eqnarray*}
\sum_{n\geq m}|\bbE_P\mu_{\te^n\om}(\cA_{it}^{\om,n}f_{\om})|\leq
\sum_{n\geq m}\bbE_P\|f_\om\|\cdot\|\cA_{it}^{\om,n}\|\\\leq
\sum_{n\geq m}\int_{\Del_n^c}\|f_\om\|\cdot\|\cA_{it}^{\om,n}\|dP(\om)+
4B_0\sum_{n\geq m}2^{-un}\int_{\Del_n}\|f_\om\|dP(\om)\\\leq
B_J\sum_{n\geq m}\int_{\Del_n^c}\|f_\om\|dP(\om)
+4B_0\bbE_P\|f_\om\|\sum_{n\geq m}2^{-un}
\end{eqnarray*}
where $\Del_n^c=\Om\setminus\Del_n$ and $B_J$ is defined after (\ref{basic nor, bound, large t's}).
Set $\|f\|_{p,\al,\xi}^p:=\int \|f_\om\|^p dP(\om)<\infty$.
Applying the H\"older inequality yields
\[
\int_{\Gam_n^c}\|f_\om\|dP(\om)\leq d^{1-\frac1p}\|f\|_{p,\al,\xi}n^{-\be(1-\frac 1p)}.
\]
Since $\be(1-\frac 1p)>1$ the series defined by the above right hand side (with $n=1,2,...$)
converges, and hence the (uniform) limit
of $L_\rho(t)$ as $\rho\to 1$ exists on $[-\del_0,\del_0]$.

Next, in the lattice case we proceed in a slightly different way. We first rewrite
(\ref{I}) as 
\[
U_\rho(g)=\frac1{2\pi}\int_{-\frac\pi h}^{\frac\pi h}r(t)L_\rho(t)dt
\]
where we used that $L_\rho$ is $\frac{2\pi}{h}$ periodic which holds true since
\[
\mu_{\te^n\om}(\cA_{it}^{\om,n}f_\om)=\bbE[f_\om(\xi_n^{\om})e^{itS_n^{\te^n\om}}]
\]
which is indeed a $\frac{2\pi}{h}$-periodic function of $t$  in view of the lattice assumption.
The proof proceeds now in a similar way taking compact subintevrals  $J$  of 
$(-\frac{2\pi}h,\frac{2\pi}h)$ which do not contain the origin.
\end{proof}

Now, we prove
\begin{lemma}\label{VII.4}
If $\del_0$ is sufficiently small then
there exists an integrable function $R_3$ on $[-\del_0,\del_0]\setminus\{0\}$ so that
with $g_a(x)=g(x-a)$, for any $a\in\bbR\setminus\{0\}$,
\[
\lim_{\rho\to1}2\pi V_\rho(g_a)=\int_{-\del_0}^{\del_0}e^{-ita}R_3(t)dt+
\frac{\boldsymbol{\mu}(0)(f)}{\gam}\ka_h(g)\big(\pi+\int_{-\del_0a}^{\del_0a}\frac{\sin t}t dt\big)
\]
where in the lattice case $\ka_h$ is the measure assigning unit mass to each point 
of the lattice $h\bbZ$, while in the non-lattice case $\ka_0$ is the Lebesgue measure.
\end{lemma}

\begin{proof}
First, we have $\hat g_a(t)=e^{-ita}\hat g(t)$ and $\sum_{k}\hat g_a(t+\frac{2\pi k}h)=e^{-ita}r(t)$.
 Observe that
\[
\rho\la_{\te^{-1}\om}(it)f_\rho^\om(t)=f_{\rho}^{\te^{-1}\om}(t)-\la_{\te^{-1}\om}(it)\mu_\om(h_\om(it))
\]
and therefore,
\begin{equation}\label{f rho repres}
f_\rho^\om(t)
=\frac{\la_{\te^{-1}\om}(it)\mu_\om(h_\om(it))}{1-\rho\la_{\te^{-1}\om}(it)}+
\frac{f_\rho^\om(t)-f_{\rho}^{\te^{-1}\om}(t)}{1-\rho\la_{\te^{-1}\om}(it)}.
\end{equation}
Since $\la_\om'(0)=\gam>0$ we obtain from (\ref{UnifBound}) that there exists a
constant $a_0>0$ so that
for any sufficiently small $\del_0$,  $t\in[-\del_0,\del_0]$ and $\frac12<\rho\leq 1$,
\begin{equation}\label{lam Tay}
|1-\rho\la_{\te^{-1}\om}(it)|\geq a_0|t|
\end{equation}
and so, the above decomposition of $f_\rho^\om(t)$ makes sense when $\rho>\frac12$
and $t\not=0$. 
Set 
\[
\cD_{\om,n}(t):=\la_{\om,n}(it)\mu_{\te^n\om}(h_{\te^n\om}(it))-
\la_{\te^{-1}\om,n}(it)\mu_{\te^{n-1}\om}(h_{\te^{n-1}\om}(it)).
\]
Then
\[
f_\rho^{\om}(t)-f_\rho^{\te^{-1}\om}(t)=\sum_{n\geq 1}\rho^{n-1}\cD_{\om,n}(t).
\]
Write
\begin{eqnarray*}
\cD_{\om,n}(t)=\la_{\om,n}(it)\big(\la_{\te^{n-1}\om}(it)\big)^{-1}\times\\
\big(\mu_{\te^n\om}(h_{\te^n\om}(it))\la_{\te^{n-1}\om}(it)-
\mu_{\te^{n-1}\om}(h_{\te^{n-1}\om}(it))\la_{\te^{-1}\om}(it)\big).
\end{eqnarray*}
As in the proof of Lemma 7.2.1 in \cite{book}, differentiating the equality
$\nu_\om(z)(h_\om(z))$ at $z=0$ and using that $h_\om(0)\equiv\textbf{1}$ and $\nu_\om(z)\textbf{1}=1$ yields
that $\boldsymbol{\mu}_\om(0)(h_\om'(0))=\nu_\om(0)(h_\om'(0))=0$.
By (\ref{UnifBound}) and since $\la_\om'(0)=\gam>0$,  when $\del_0$ is sufficiently
small then the absolute value of the term inside the brackets above does not exceed 
$c_0t^2$ for some constant $c_0$ which does not depend on $\om,n$ and 
$t$. Indeed, this term  has bounded second derivatives, it vanishes at $t=0$ and
its first derivative at $0$ equals $0$. Now, by (\ref{lam Tay})
there exists a constant $c_1>0$  so that $|\la_{\te^{n-1}\om}(it)|\geq c_1$
for any $n\geq 1$ and $t\in[-\del_0,\del_0]$. 
We conclude that there exist constants $c_2$ and $c_2'$ so that for any $t\in[-\del_0,\del_0]$, $1\leq q<\infty$ and $n\geq1$,
\begin{equation}\label{Good Exp1}
\|\cD_{\om,n}(t)\|_q
\leq c_2't^2\|\la_{\om,n}(it)\|_q\leq c_2t^2(P(\Gam_n^c))^{\frac1q}
+c_2t^2e^{-nd_2t^2}
\end{equation}
where $\|X\|_q=\|X\|_{L^q(\Om,\cF,P)}$ for any random variable $X$, the set $\Gam_n^c$ is the compliment of the set
satisfying (\ref{First Gamma props1}) and (\ref{First Gamma props2})
 and we also used (\ref{lam basic bound}). 
In particular,  by the H\"older inequality for any $Y\in L^p(\Om,\cF,P)$, 
$1<p\leq\infty$,
\begin{equation}\label{Good Exp2}
E|Y\cD_{\om,n}(t)|\leq \|Y\|_p\big(c_2t^2(P(\Gam_n^c))^{1-\frac1p}
+c_2t^2e^{-nd_2t^2}\big).
\end{equation}
Since $\be(1-\frac1p)>1$ and $t^2\sum_{n\geq1}e^{-nd_2t^2}$ is bounded in $t\in[-\del_0,\del_0]$, we obtain from (\ref{Good Exp2}) that
\begin{equation}\label{Good Exp3}
\bbE_P|Y(f_\rho^{\om}(t)-f_\rho^{\te^{-1}\om}(t))|\leq C(1+\|Y\|_p)
\end{equation}
where $C>0$ is a constant which does not depend on $Y$
and $t$.
Next, for any $t\in[-\del_0,\del_0]$ set $\psi^\om(t)=e(t)\la_{\te^{-1}\om}(it)
\mu_\om(h_\om(it))\nu_\om(it)(f_\om)$
where $e(t)=\hat g(t)$ in the non-lattice case, while in the lattice case
$e(t)=r(t)$.  
Then by (\ref{UnifBound}) we have $|\psi^\om(t)|\leq C'\|f_\om\|_{\al,\xi}$, where $C'>0$ is some constant.
Moreover,
\[
\psi^\om(0)=e(0)\nu_\om(0)(f_\om)=\nu(0)(f)\int g(x)d\ka_h(x)
\]
and $\psi^\om(t)=\psi^\om(0)+t\psi_1^\om(t)$ for some bounded function $\psi_1^\om:[-\del_0,\del_0]\to\bbC$
so that $\bbE_P|\psi^\om(t)|\leq C''\bbE_P\|f_\om\|_{\al,\xi}$ for any $t\in[-\del_0,\del_0]$, where $C''$ is another constant.
We also set 
$\Del_\rho^\om(t)=f_\rho^\om(t)-f_\rho^{\te^{-1}\om}(t)$
and $\phi^\om(t)=e(t)\nu_\om(it)(f_\om)$.
Then by (\ref{Good Exp3}) the random variables $\Del_\rho^\om(t)$ are integrable and since
$P$ is $\te$-invariant we have $\bbE_P\Del_\rho^\om(t)=0$.
Now, for any $\rho\in(\frac12,1)$ write,
\[
2\pi V_\rho(g_a)=\int_{-\del_0}^{\del_0}e^{-ita}(F_\rho(t)+G_\rho(t))dt+J(a,\rho)\ka_h(g)\nu(0)(f)
\]
where
\begin{eqnarray*}
J(a,\rho)=\int_{-\del_0}^{\del_0}\frac{e^{-ita}}{1-\rho(1+i\gam t)}dt,\\
F_\rho(t)=\bbE_P\big[\frac{\psi^\om(t)}{1-\rho\la_{\te^{-1}\om}(it)}-\frac{\psi^\om(0)}{1-\rho(1+i\gam t)}\big]\\
\text{and }\,\,G_\rho(t)=\bbE_P\big[\frac{\phi^\om(t)\Del_\rho^\om(t)}{1-\rho\la_{\te^{-1}\om}(it)}\big]
\end{eqnarray*}
and we used (\ref{f rho repres}).
As in the proof of Lemma VII.3 in \cite{HenHerve} we have 
\[
\lim_{\rho\to1}J(a,\rho)=\gam^{-1}\big(\pi+\int_{-a\del_0}^{a\del_0 }\frac{\sin t}t dt\big).
\]
In order to compete the proof of the lemma in the case discussed above,
it is sufficient to show that $F_\rho(t)$ and $G_\rho(t)$ 
are bounded in both $t\in[-\del_0,\del_0]\setminus\{0\}$ and $\rho\in(-\frac12,1)$ and that the pointwise limits
 $\lim_{\rho\to1}F_\rho(t)$ and  $\lim_{\rho\to1}G_\rho(t)$ exist on $[-\del_0,\del_0]\setminus\{0\}$.
We first show that the above statement holds true for $F_\rho(\cdot)$. 
First, by (\ref{lam Tay}) and since $|\psi^\om(t)|\leq C'\|f_\om\|_{\al,\xi}$, 
applying the dominated convergence theorem yields
that for any $t\in[-\del_0,\del_0]\setminus\{0\}$,
\[
\lim_{\rho\to1}F_\rho(t)=\bbE_P\big[\frac{\psi^\om(t)}{1-\la_{\te^{-1}\om}(it)}
-i\frac{\psi^\om(0)}{\gam t}\big].
\]
Moreover, for any $t$ and $\rho$ in the above domain,
\begin{eqnarray*}
|F_\rho(t)|=\Big|t\bbE_P\big[\frac{\psi_1^\om(t)}{1-\rho\la_{\te^{-1}\om}(it)}\big]+
\rho\bbE_P\big[\frac{\la_{\te^{-1}\om}(it)-1-i\gam t}{\big(1-\rho\la_{\te^{-1}\om}(it)\big)\big(1-\rho(1+i\gam t)\big)}\psi^\om(0)\big]\Big|
\\\leq C_0\bbE_P\|f_\om\|_{\al,\xi}+C_0|e(0)\nu(0)(f)|t^{-2}\bbE_P|\la_{\te^{-1}\om}(it)-1-i\gam t|\\
\leq A_0(\bbE_P\|f_\om\|_{\al,\xi}+|e(0)\nu(0)(f)|)
\end{eqnarray*}
where we used (\ref{UnifBound}) and (\ref{lam Tay}) and
$C_0,A_0>0$ are some constants, and we obtain the desired estimate.

Next, in order to show that $G_\rho$ converges pointwise on $[-\del_0,\del_0]\setminus\{0\}$ as $\rho\to1$, we will need the following simple result. Let $X_1,X_2,...$ be
a sequence of random variables so that $\sum_{n\geq1}\bbE|X_n|<\infty$. Let $Y_\rho,\,\rho\in(\frac12,1]$ be a collection of random variables so that $\|Y_\rho-Y_1\|_{L^\infty}$ converges to $0$ as $\rho\to1$ and $Y_1$ is bounded. Then 
\[
\lim_{\rho\to1}\sum_{n\geq 1}\rho^{n-1}\bbE X_nY_\rho=\sum_{n\geq 1}\bbE X_nY_1. 
\]
Using (\ref{Good Exp2}) with $Y\equiv 1$  and then the above result with 
$X_n=\phi^\om(t)\cD_{\om,n}(t)$ and $Y_\rho=\frac1{1-\rho\la_{\te^{-1}\om}(it)}$ we derive that the limit $\lim_{\rho\to1}G_\rho(t)$ exists for any $t$ in the above domain. To complete the proof, we will show that $G_\rho(t)$ is bounded
when considered as a function of $(\rho,t)\in(\frac12,1)\times[-\del_0,\del_0]$. Indeed, let $(\rho,t)$ be
in the latter domain. In the following arguments all the constants will depend only on $\del_0$ (and not on $\om,\rho$
and $t$). Since $\bbE_P\Del_\rho^\om(t)=0$ and $\phi_\om(0)$ does
not depend on $\om$, for we can write
\begin{eqnarray*}
G_\rho(t)=\bbE_P\frac{\phi^\om(t)\Del_\rho^\om(t)}{1-\rho\la_{\te^{-1}\om}(it)}\\=
\bbE_P\big[\phi^\om(t)\Del_\rho^\om(t)\big((1-\rho\la_{\te^{-1}\om}(it))^{-1}-(1-\rho(1+i\gam t))^{-1}\big)\big]\\
+\bbE_P\big[\big(\phi^\om(t)\Del_\rho^\om(t)-\phi^\om(0)\Del_\rho^\om(t)\big)(1-\rho(1+i\gam t))^{-1}\big]:=I_1+I_2.
\end{eqnarray*}
By (\ref{Good Exp3}), (\ref{UnifBound}) and our assumption that $\be(1-\frac1p)>1$ we have 
\[
\bbE_P|\phi^\om(t)\Del_\rho^\om(t)|\leq C_1\bbE_P[\|f_\om\|_{\al,\xi}|\Del_\rho^\om(t)|] \leq C_2
\]
for some constants $C_1, C_2>0$. Moreover, by (\ref{UnifBound}) and (\ref{lam Tay}) 
there exist constants $C_3>0$ and $C_3'>0$
so that
\[
|(1-\rho\la_{\te^{-1}\om}(it))^{-1}-(1-\rho(1+i\gam t))^{-1}|\leq C_3't^{-2}|\la_{\te^{-1}\om}(it)-(1+i\gam t)|
\leq C_3.
\]
We conclude that $|I_1|\leq C$ for some constant $C$. 
Next, by (\ref{UnifBound}), there exists a constant $C_4>0$ so that $|\phi^\om(t)-\phi^\om(0)|
\leq C_4(1+\|f_\om\|_{\al,\xi})|t|$, and therefore, using also (\ref{Good Exp3}),  
\[
|I_2|=|1-\rho(1+i\gam t)|^{-1}
\big|\bbE_P\big[(\phi^\om(t)-\phi^\om(0))\Del_\rho^\om(t)\big]\big|\leq C_5
\]
where $C_5$ is another constant. The proof of the lemma is complete
now in the case when $\mu_\om(h'_\om(0))$ does not depend on $\om$.

\end{proof}

The completion of the proof of Theorem \ref{renewal-D.S.} is done similarly 
to \cite{HenHerve} (see references therein).
Let $g\in\mathscr H_1$ be dominated by a positive member of $\mathscr H_1$. If $g$ is positive then 
$U(g)=\lim_{\rho\to1}U_\rho(g)$ which is finite in view of Lemmas 
\ref{VII.3} and \ref{VII.4}. Thus $U$ is a Radon measure on $\bbR$
(see \cite{HenHerve} and references therein).
Next,  for any $a\in\bbR$ the integral $U(g_a)$ is defined and is the limit of 
$U_\rho(g_a)$ as $\rho\to1$. Therefore by Lemmas \ref{VII.3} and \ref{VII.4},
\begin{equation}\label{RHS}
U(g_a)-\frac{\nu(0)(f)\ka_h(g)}{\gam}\frac1{2\pi}\big(\pi+\int_{-a\del_0}^{a\del_0}\frac{\sin t}{t}dt\big)=\widehat{Q_1}(a)+\widehat{Q_2}(a)+
\frac1{2\pi}\widehat{R_3}(a)
\end{equation}
where $Q_i(t)=R_i(t)e(t),\,i=1,2$ and
we set $R_1(t)=R_3(t)=0$ for any $t\in(-\del_0,\del_0)$ and $R_2(t)=0$ for any 
$t\in\bbR\setminus[-\del_0,\del_0]$. The functions $Q_1$, $Q_2$ and $R_3$ are integrable,
and so, by the Riemann-Lebesgue Lemma,  the right hand side of (\ref{RHS}) converges to $0$
as $|a|\to\infty$. Finally, as in \cite{HenHerve}, 
\[
|\bbI(a\geq 0)-\frac1{2\pi}\big(\pi+\int_{-a\del_0}^{a\del_0}\frac{\sin t}{t}dt\big)|\leq 2|a|^{-1}
\]
for any nonzero $a\in\bbR$, and the proof of Theorem \ref{renewal-D.S.} is complete.






\section{Mixing conditions}\label{MixSec}\setcounter{equation}{0}


In this section we will show that conditions (\ref{V-mix}) and (\ref{Neighb mix}) 
in Assumption \ref{theta mixing assumption} hold true for shift spaces generated 
by certain mixing processes and for (natural invertible extensions of) certain dynamical 
systems. We will focus only the invertible case, and the conditions will hold true 
in the non-invertible case by considering functions of the $0$-the coordinate (in the natural extension)
and using Section \ref{Reduction}. We will always assume that Assumption \ref{bound ass+Tr Op meas+ass phi u} holds true.
We begin with
\begin{proposition}\label{Prop1}
Suppose  that there exists a measure
$\mu$ on $\cX$ and functions $r_\om:\cX\to\bbR$ so that $\boldsymbol{\mu}_\om(0)=r_\om \mu$ and 
all $r_\om$'s take values at some finite interval $(m,M)\subset(0,\infty)$. Moreover, assume that there exist
 constants $d_1>0$ and $\be>0$ so that for any sufficiently large $k$ and $n\geq1$,
\begin{equation}\label{Vk mix}
P\{\om:\,\sum_{j=0}^{n-1}\tilde V_k(\te^{jk}\om)\leq d_1kn\}\leq\frac{d_2(k)}{n^\be}
\end{equation}
where $\tilde V_k(\om)=\bbE_\mu (S_k^\om \bar u)^2$
and $d_2(k)$ is a constant which depends only on $k$.
Then (\ref{V-mix}) holds true. 
In particular, (\ref{V-mix}) holds true when polynomial
concentration inequalities of the form
\begin{equation}\label{Poly Vk}
P\{\om: |\sum_{n=0}^{n-1}\tilde V_k(\te^{jk}\om)-n\bbE_P\tilde V_k|\geq \ve n\}
\leq\frac{d(k,\ve)}{n^\be} 
\end{equation}
hold true, where $\ve>0$ and $d(k,\ve)$ is a constant which depends only $k$ and $\ve$.
\end{proposition}

\begin{proof}
Let $1\leq k\leq n$, set 
\[
S(\om,j)=S_{jk}^{\te^{jk}\om}\bar u,\,\,j=0,1,...,[n/k]-1
\]
and 
$S(\om,[n/k])=S^{\te^{[n/k]}\om}_{n-k[n/k]}\bar u$, where $S_0^\om$ is defined to be $0$. 
Since $\|u_\om\|_{\al,\xi}$ is a bounded random variable and $\nu(0)=\boldsymbol{\mu}(0)$ 
is $T$-preserving, applying 
Lemma 5.10.4 in \cite{book} together with (\ref{Exp Conv final.0}), 
namely, the uniform in $\om$ exponential decay of correlations, yields
that 
\begin{eqnarray*}
V_n(\om)=\bbE_{\nu_\om(0)}\big(\sum_{j=0}^{[n/k]}S(\om,j)\big)^2
=\sum_{0\leq j_1,j_2\leq [n/k]}\text{cov}_{\nu_\om(0)}(S(\om,j_1), S(\om,j_2))\\
\geq \sum_{j=0}^{[\frac nk]-1}\bbE_{\nu_\om(0)}(S(\om,j))^2-A(\frac{n}{k}+1)=
\sum_{j=0}^{[\frac nk]-1}V_k(\te^{jk}\om)-A(\frac{n}{k}+1)
\end{eqnarray*}
where $A>0$ is some positive constant. It follows from the assumption about
the densities $r_\om$ that
\[
V_n(\om)\geq\frac{m}{M}\sum_{j=0}^{[\frac nk]-1}\tilde V_k(\te^{jk}\om)-A(\frac{n}{k}+1).
\]
Taking a sufficiently large $k$ we derive that on the complement of the set whose 
probability is estimated in (\ref{Vk mix}) we have
\[
V_n(\om)\geq \frac{Md_1n}{m}-A(\frac{n}{k}+1).
\]
If $k$ is sufficiently large then  the above right hand side is not less than $cn$ for some 
constant $c>0$, and so condition (\ref{V-mix}) is satisfied.
Finally, (\ref{Poly Vk}) implies (\ref{Vk mix}) since $\bbE_P\tilde V_k\geq \frac{m}{M}\bbE_P V_k$ and $\lim_{k\to\infty}k^{-1}\bbE_P V_k=\sig^2>0$ 
\end{proof}
Note that when $\boldsymbol{\mu}_\om=\nu_\om(0)$ does not depend on $\om$ then the first assumption in 
Proposition is satisfied with $\mu=\nu_\om(0)$ and $r_\om=1$. The measures $\boldsymbol{\mu}_\om$ do not depend on $\om$ when only the function $u_\om$ is random.
These measures do no depend on $\om$ also in the non-invertible case considered in Section \ref{Non Inv Sec}, since, in the extension, we have $\nu_\om(0)=\boldsymbol{\mu}_\om(0)=\ka$.
When $\boldsymbol{\mu}_\om$ depends on $\om$, then, in the circumstances of Proposition \ref{Prop2},
the first assumption is satisfied with $r_\om=\textbf{h}_\om(0)$ and $\mu=\textbf{m}$, since the function  
function $\textbf{h}_\om(0)$ is bounded from above and
below by positive constants not depending on $\om$ (see Section 5.12 in \cite{book}).

Henceforth, we assume that all the conditions of Proposition \ref{Prop1} are satisfied.
We provide now sufficient conditions for (\ref{Poly Vk}) and (\ref{Neighb mix}) to hold true
in two situations. First consider the case when $(\Om,\cF,P,\te)=(\hat{\Om}_0,\hat{\cF_0},\hat{P_0},\hat\vartheta)$ is the natural invertible extension
(described in \ref{Reduction}) of a measure preserving system $(\Om_0,\cF_0,P_0,\vartheta)$. Note that $\te$
has a periodic point if and only if $\vartheta$ has a periodic point.
We assume here that $\Om_0$ is equipped with a metric $d_0$ so that $\text{diam}\Om_0\leq1$ and $\cF_0$ contains the appropriate
Borel $\sig$-algebra.
 Let the metric $d$ on $\Om_0^\bbZ$ be defined by 
\[
d(a,b)=\sum_{n\in\bbZ}2^{-|n|}d_0(a_n,b_n)\,\,
\text{ for any }\,\,a=(a_n)_{n\in\bbZ}\,\text{ and }\, b=(b_n)_{n\in\bbZ}.
\] 
Let $\boldsymbol{\iota}:\Om\to\Om_0^\bbZ$ be the  inclusion map given by $\boldsymbol{\iota}\om=\om$.
We also need the following
\begin{assumption}\label{MixAssNew}
The function $u_\om$ and the transformation $T_\om$ have the form 
$u_\om=\boldsymbol{u}_{\boldsymbol{\iota}\om}$ and $T_\om=\boldsymbol{T}_{\boldsymbol{\iota}\om}$,
where $\boldsymbol{u}_{v}$ and $\boldsymbol{T}_v$ are H\"older continuous functions of
the variable $v\in(\Om_0^\bbZ,d)$ with respect to the norm $\|\cdot\|_{\al,\xi}$ and the 
metric $d(\boldsymbol{T}_{v_1},\boldsymbol{T}_{v_2})=\sup_{x\in\cX}
\rho(\boldsymbol{T}_{v_1}x,\boldsymbol{T}_{v_2}x)$, respectively.
\end{assumption}
This assumption includes the case when $u_\om$ and $T_\om$ depend only on the first coordinate $\zeta_0$ and 
are H\"older continuous functions of this coordinate when considered as functions on 
$(\Om_0,d_0)$ (as in Section \ref{Reduction}).
Under Assumption \ref{MixAssNew}, the functions $\tilde V_k,\,k\geq1$ have the form 
$\tilde V_k(\om)=\boldsymbol{\tilde V}_k(\boldsymbol{\iota}\om)$, where $\boldsymbol{\tilde V}_k$ is 
a H\"older continuous function on $(\Om_0^\bbZ,d)$. Therefore, for any $N>0$ there exists a H\"older continuous function $\boldsymbol{\tilde V}_{k,N}$, which depends only on the coordinates whose indexes lie in $\{-N,...,0,...N\}$, 
so that 
\begin{equation}\label{Approx}
\sup_v|\boldsymbol{\tilde V}_{k}(v)-\boldsymbol{\tilde V}_{k,N}(v)|\leq C_k2^{-N}
\end{equation}
 where $C_k$ is some 
constant which depends only on $k$.

 
Relying on (\ref{Approx}) and the results in \cite{Chaz} (see also \cite{Chaz2}) and \cite{Haf-MD},
in the above circumstances we have

\begin{proposition}\label{Prop3}
Let $\beta>0$. Then conditions (\ref{Vk mix}) and (\ref{Neighb mix}) hold true 
when $(\Om_0,\cF_0,P_0,\vartheta)$ is 
either a topologically mixing subshift of finte type, a Young tower with at least one periodic point,
with tails of order $n^{-\beta-1}$
or when $\vartheta^n(\om)=\xi_n(\om)$, where $\xi_0,\xi_1,\xi_2,...$ is the stationary vector valued processes satisfying the mixing and approximation conditions from \cite{Haf-MD}, assuming that $\vartheta$ has at least one 
periodic point, $\Om_0$ is a metric space and $P_0(A)>0$ for all open sets $A$.
\end{proposition}
This proposition holds true since (by either \cite{Chaz} or \cite{Haf-MD}) all the maps mentioned there satisfy (\ref{Poly Vk}) and 
that for any Lipschitz function $f:\Om_0^\bbZ\to\bbR$ there exist constants 
$d(\ve),\ve>0$ so that for any $n\geq 1$,
\[
P\{\om:|\sum_{n=0}^{n-1}f(\te^{j}\om)-n\bbE_Pf|\geq \ve n\}
\leq\frac{d(\ve)}{n^\be}.
\] 
This inequality implies (\ref{Neighb mix})
since the indicator function of any ball $B(v,r)\subset\Om_0^\bbZ$
can be approximated from below by a Lipschitz function $f$ which takes the constant value $1$
on $B(v,\frac12 r)$ (see, for instance, Section 1.2.9 in \cite{book}).
 
Note that when $\vartheta$ is a Young tower the requirement of having a peridoic point is just 
the assumption that the function $\vartheta^R|\Gam:\Gam\to\Gam$ has at least one periodic point, which is not too restrictive. Here $\Gam$ is the (hyperbolic) base set and $R$ is the corresponding return time function.

Next, let $\zeta=\{\zeta_n,\,n\geq 0\}$ be a stationary sequence of random variables defined on 
a probability space $(\Om_0,\cF_0,P_0)$ taking values on 
some  metric space $(\cY,d_0)$ so that $\text{diam}\cY\leq1$. Let $(\Om,\cF,P,\te)$ be
the natural associated shift, namely, $\Om=\cY^\bbZ$, $\cF$ is the product $\sig$-algebra, $P$ is 
given by 
\[
P\big\{(y_j)_{j\in\bbZ}: y_i\in A_i;\,|i|\leq N\big\}=P_0(\zeta_i\in A_{i-N};\,0\leq i\leq 2N)
\]
and $\te$ is the two sided shift. 
Let $d$ be the metric on $\Om$ given by 
$d(\{a_n\},\{b_n\})=\sum_{n\in\bbZ}2^{-|n|}d_0(a_n,b_n)$. We assume here
that $\zeta$ is stretched exponentially 
$\al$-mixing, namely that there exist constants $a,b,c>0$ so that 
for any $n\geq1$, $k\geq1$, $A\in\sig\{\zeta_1,...,\zeta_k\}$ 
and $B\in\sig\{\zeta_{i}:\, i\geq k+n\}$,
\begin{equation}\label{Mix}
\big|P_0(A\cap B)-P_0(A)P_0(B)|\leq ae^{-bn^c}. 
\end{equation}
We first have

\begin{proposition}\label{Prop4}
(i) Suppose that the functions $\om\to u_\om$ and $\om\to T_\om$ are H\"older continuous
with respect to the norm $\|\cdot\|_{\al,\xi}$ and the metric 
$d(T_{\om},T_{\om'})=\sup_{x\in\cX}\rho(T_{\om}x,T_{\om'}(x))$, respectively. Then the functions 
$\tilde V_k$ are H\"older continuous and condition (\ref{Poly Vk}) holds 
true with any $\beta>0$.

(ii) Let $\cJ\subset\bbZ$ be a finite 
set and let $\pi_{\cJ}:\Om\to\cY^{J}$ be the projection corresponding to the coordinates indexed by members
of $\cJ$. Then condition (\ref{Poly Vk}) also holds true (with any $\beta>0$) when 
$u_\om=u_{\pi_{\cJ}\om}$ 
and $T_\om=T_{\pi_{\cJ}\om}$ depend only on the coordinates in places indexed by the members of $J$, without assuming that $\tilde V_k$'s are continuous.
\end{proposition}
Proposition \ref{Prop4} follows from the so-called method of cumulants, 
see, for instance, \cite{SaulStat}, \cite{KalNe}, \cite{Douk}, \cite{Ded} or
\cite{Haf-MD} (in the case $\ell=1$).
 The mixing condition (\ref{Mix})
holds true, for instance, when $\zeta_n=\zeta_0\circ\vartheta^n$ and  $\zeta_0$ is measurable with respect to a Markov 
partition corresponding to either a Young tower with stretched exponentially tails (see \cite{Hyd} or 
\cite{Hyd1} for verification of (\ref{Mix})) or a topologically mixing subshift of finite type (see \cite{Bow}), 
or has the form $\zeta_n=f(\Upsilon_n)$ when $\Upsilon_n,\,n\geq1$ is a geometrically 
ergodic Markov chain or a Markov chain satisfying the Doeblin condition 
(see \cite{BHK}). Note that the cases discussed in Proposition \ref{Prop4} include the case when
$T_\om^n=T_{\zeta_n}\circ T_{\zeta_{n-1}}\circ\cdots\circ T_{\zeta_1}$, namely the case when the compositions of the maps $T_\om$ are taken along stationary and sufficiently fast mixing process.

In the above circumstances existence of a periodic points is trivial 
(see the last paragraph of Section \ref{sec1}), and the question is whether (\ref{Neighb mix})
holds true. 

\begin{proposition}\label{Prop5}
Condition (\ref{Neighb mix}) holds true with $\om_0=(...a,a,a,...),\,a=(a_i)\in\cY^{n_0}$ when 
\[
P_0(\zeta_{i+(j-1)n_0}\in A_{i,j};\,\,i=0,1,...,n_0-1,\,1\leq j\leq s)>0
\] 
for any open sets $A_{i,j}$ so that  $a_i\in A_{i,j}$ for all $i$ and $j$.
\end{proposition}
Proposition \ref{Prop5} holds true since in its circumstances (by the method of cumulants), for any open set $B$ which depends only finite number of coordinates 
and $\beta>0$, there exist positive constants $c_\be(\ve),\,\ve>0$ so that for any $n\geq 1$,  
\begin{equation}
P\{\om: |\sum_{j=0}^{n-1}\bbI_{B}(\te^j\om)-nP(B)|\geq \ve n\}
\leq \frac{c_\be(\ve)}{n^\be}.
\end{equation}
In particular, we can consider  Markov chains with finite number of states
and more generality
Markov chains with positive densities $p(x,y)$ around $(a_i,a_{i+1}),i=1,...,n_0-1$,
whose stationary measure assigns positive mass to open sets. The proposition also holds true
when $\zeta_0$ is measurable
with respect to an appropriate Markov partition since then the non-empty intersection
$\cap_{0\leq i<n_0}\cap_{1\leq j\leq s}\vartheta^{-(j-1)n_0+i}A_{i,j}$ has positive $P_0$-measure.

\section{Additional results}

\subsection{Non-continuous functions}\label{Non-Iden}
We explain here how to obtain all the results stated in Section \ref{sec2} when $\phi_\om$ and $u_\om$ are H\"older continuous only
on some pieces of $\cX$. First, under Assumption \ref{Ass pair prop},\,$P$-a.s. for any $n\geq1$ and 
$x,x'\in\cX$ with $\rho(x,x')<\xi$ 
we can write 
\begin{equation}\label{Pair1}
(T_\om^n)^{-1}\{x\}=\{y_1,....,y_k\}\,\,\text{ and }
\,\,(T_\om^n)^{-1}\{x'\}=\{y_1',...,y_k'\}
\end{equation}
where 
\[
k=k_{\om,x,n}=|(T_\om^n)^{-1}\{x\}|\leq D_{\om,n}:=\prod_{i=0}^{n-1}D_{\te^i\om},
\]
$|\Gam|$ denotes the cardinality of a finite set $\Gam$
and with  $\gam_{\om,i}=\prod_{s=0}^{i-1}\gam_{\te^s\om}$,
\begin{equation}\label{Pair2}
\rho\big(T_\om^jy_i,T_\om^jy_i'\big)
\leq (\gam_{\te^j\om,n-j})^{-1}\rho(x,x')
\end{equation}
for any $1\leq i\leq k$ and $0\leq j<n$.

Let $H_\om$ be a random variable and let $\psi_\om\in\cH^{\al,\xi}$ be so that 
$v_{\al,\xi}(\psi_\om)\leq H_\om$.
Then by Lemma 5.1.4 in \cite{book}, for any $n\geq 1$, 
$x,x'\in\cX$ with $\rho(x,x')<\xi$ and $1\leq i\leq k$, 
\begin{equation}\label{S n distortion}
|S_n^\om\psi(y_i)-S_n^\om\psi(y_i')|
\leq\rho^\al(x,x')\sum_{j=0}^{n-1}H_{\te^j\om}
(\gam_{\te^j\om,n-j})^{-\al}
\end{equation}
where $y_1,...,y_k$ and 
$y_1',...,y_k'$ satisfy (\ref{Pair1}) and (\ref{Pair2}).
Not only members of $\cH_{\al,\xi}$ satisfy (\ref{S n distortion}).
For instance, when $\cX$ is is a $C^2$-compact Riemanian manifold 
and there exist a finite collection of disjoint rectangles $\{I_j\}$ so that $T_\om|{I_{j}}:I_j\to\mathbb{S}^1$ is an expanding diffeomorphism for each $j$,
then (\ref{S n distortion}) also holds true with some $H_\om$  for functions $\psi_\om$ which are only H\"older
continuous when restricted to each one of the $I_j$'s. Perhaps the most interesting case is when 
$\psi_\om=-\ln\big(\frac{d(T_\om)_*\textbf{m}}{d\textbf{m}}\big)$ where $\textbf{m}$ is the 
normalized volume measure. This includes, of course, the case when
$\cX=\mathbb{S}^1$, $\textbf{m}$ is the Lebesgue measure and $I_j$'s are disjoint arcs 
(namely, the classical case of random distance expanding maps on the unit interval).

 We have the following
\begin{theorem}
Suppose that all the conditions of Theorems \ref{CLT}, \ref{LLT--D.S.} and \ref{renewal-D.S.}
hold true, except for the ones concerning $\|\phi_\om\|_{\al,\xi}$ and $\|u_\om\|_{\al,\xi}$.
Assume that $\phi_\om$ and $u_\om$ satisfy (\ref{S n distortion}) with some bounded random variable $H_\om$. 
Then all the results stated in  Theorems \ref{CLT}, \ref{LLT--D.S.} and \ref{renewal-D.S.} hold true.
\end{theorem}
This theorem is proved exactly as in Section \ref{sec4}, since
all the results from \cite{book}  that we applied hold true when 
(\ref{S n distortion}) holds with bounded $H_\om$'s.
In this case, the continuity condition in Assumption \ref{PerPointAss} will
be satisfied when the maps
$\om\to \phi_\om,u_\om$ are continuous with respect to the supremum norm 
and the differences $S_n\phi_{\om_1}-S_n\phi_{\om_2}$ and 
$S_n u_{\om_1}-S_n u_{\om_2}$ satisfy (\ref{S n distortion}) with some constant $H=H(\om_1,\om_2)$
so that $\lim_{\om_1,\om_2\to\te^{k}\om_0} H(\om_1,\om_2)=0,\,k=0,1,...,n_0-1$. 
When $T_\om$ is locally constant around the points in the orbit of the periodic point
$\om_0$ then,
in the examples discussed after 
(\ref{S n distortion}),  this condition means that  restrictions of $\phi_{\om}$
and $u_{\om}$ to each one of the $I_j$'s is a continuous 
function of $\om$ at $\om=\te^j\om_0,\,0\leq j<n_0$ with respect to the H\"older norm $\|\cdot\|_{\al,\xi}$.

\subsection{Non-identical fibers}\label{NonFib}
Let $\cE\subset\cF\times\cB$ be a measurable set
such that the fibers $\cE_\om=\{x\in \cX:\,(\om,x)\in\cE\},\,\om\in\Om$ are compact. The latter yields (see \cite{CV} Chapter III) that the mapping $\om\to\cE_\om$ is measurable with respect to the Borel $\sig$-algebra induced by the Hausdorff topology on the space $\cK(\cX)$ of compact subspaces of $\cX$ and the distance function $d(x,\cE_\om)$ is measurable in $\om$ for each $x\in \cX$.  
Furthermore, the projection map $\pi_\Om(\om,x)=\om$ is
measurable and it maps any $\cF\times\cB$-measurable set to a
$\cF$-measurable set (see ``measurable projection" Theorem III.23 in \cite{CV}).
Denote by $\cP$ the restriction of $\cF\times\cB$ on $\cE$.
Let 
\[
\{T_\om: \cE_\om\to \cE_{\te\om},\, \om\in\Om\}
\]
be a collection of continuous bijective maps between the metric spaces 
$\cE_\om$ and $\cE_{\te\om}$ so that
the map $(\om,x)\to T_\om x$ is measurable with respect to $\cP$ and each $T_\om$
is topologically exact and has the pairing property (namely, appropriate versions of 
Assumptions \ref{TopExRand} and \ref{Ass pair prop} hold true, see Chapter 5 of \cite{book} for the precise
formulations).

According to Lemma 4.11 in \cite{MSU} (applied with $r=\xi$), 
there exists an integer valued random variable $L_\om\geq1$ and $\cF$-measurable 
functions $\om\to x_{\om,i}\in\cX,\,i=1,2,3,...$ so that 
$x_{\om,i}\in\cE_\om$ for each $i$ and
\begin{eqnarray}\label{cover}
\bigcup_{k=1}^{L_\om} B_\om(x_{\om,k},\xi)=\cE_\om,
\,\,\,P\text{-a.s.}
\end{eqnarray}
Suppose that $L_\om$ is bounded. Then the proof of Theorem \ref{CLT} proceeds exactly 
as in Section \ref{sec4} since Theorem \ref{RPF rand T.O. general} and all the other 
results stated in Section \ref{sec3} hold true.

The role of the condition  that $\cE_\om$ does not depend on $\om$ in the proofs of Theorems \ref{LLT--D.S.}
and \ref{renewal-D.S.}
is only to insure that the operators $\cL_{it}^\om$ are defined on the same space when $\om$ lies
in some neighborhood of one of the members of the (periodic) orbit of $\om_0$.
The proofs of Theorems \ref{LLT--D.S.} and \ref{renewal-D.S.} proceed similarly when there 
exist open neighborhoods $U_{j}$ of $\om_j:=\te^j\om_0,\,j=0,1,...,n_0-1$ so that $\cE_\om=\cE_{\om_j}$
for any $\om\in U_j$, namely, when $\cE$ is a product set only in neighborhoods of points belonging to the 
periodic orbit of $\om_0$. In fact, the proof is carried out similarly
when for any $j$ and $\om\in U_j$ there exists a bilipschitz  homomorphism $\varphi_{\om,\om_j}:\cE_{\om_j}\to\cE_\om$, whose Lipschitz constant is bounded in $\om$, and for any compact set $J$ and $j=0,1,...,n_0-1$,
\[
\lim_{\om\to\om_j}\,\sup_{t\in J}\,\sup_{g\in\cH_\om^{\al,\xi}: \|g\|_{\al,\xi}=1}\,
\big\|\cL_{it}^\om g_\om-\big(\cL_{it}^{\om_j}(g\circ \varphi_{\om,\om_j})\big)\circ 
\varphi_{\om_{j+1},\te\om}\big\|_{\al,\xi}=0.
\] 


\subsection{Markov chains with transfer (transition) operators}
Suppose that $(\Om,\cF,P,\te)$ is invertible.
Let $\mu_\om$ be a (measurable in $\om$) probability measure on $\cE_\om$ and let $\xi_n^{\te^{-n}\om},\,n\geq0$
be a Markov chain with initial distribution $\mu_\om$ whose $n$-th step operator is given by 
$\cA_{0}^{\te^{-n}\om,n}$. Set 
\[
S_n^\om=\sum_{j=0}^{n-1}u_{\te^{-j}\om}(\xi_j^{\te^{-j}\om}).
\]
Let $S_n$ be the random variable generated by drawing $\om$ according to $P$ and taking on the fibers
the distribution of $S_n^\om$, namely the random variable whose characteristic function is given by
\[
\bbE e^{itS_n}=\int\mu_\om(\cA_{it}^{\te^{-n}\om,n}\textbf{1})dP(\om).
\]
Then the appropriate versions of Theorems \ref{CLT} and \ref{LLT--D.S.} are proved for the sequence of random variables
$S_n,\,n\geq 1$ exactly as in 
Section \ref{sec4}. As for the renewal theorem, the arguments in Section \ref{Renewal Section}
yield the following
 
\begin{theorem}\label{renewa2}
Suppose that Assumptions \ref{bound ass+Tr Op meas+ass phi u}, \ref{PerPointAss} and \ref{theta mixing assumption} hold true,
where in the last assumption we require that $\beta>1$.
Moreover, assume that
 $\boldsymbol{\mu}_\om(0)(u_\om)=\gam>0$ does not depend on $\om$.
Let $f_\om\in\cH_{\al,\xi}^\om$ be a positive function so that $\boldsymbol{\mu}_\om(0)(f_\om)=\boldsymbol{\mu}(0)(f)$ does not
depend on $\om$ and that $\|f_\om\|_{\al,\xi}\in L^p(\Om,\cF,P)$ for some 
$1<p\leq\infty$ so that $\be(1-\frac1p)>1$. For any Borel measurable set $B\subset\bbR$
set
\begin{eqnarray*}
&U(B)=U_{\mu,f}(B)=
\sum_{n\geq1}\bbE[f(S_n-S_{n-1})\bbI_B(S_n)]\\
&=\sum_{n\geq1}\int\mu_{\te^{-n}\om}
\big(f_{\te^{-n}\om}(\xi_n^{\te^{-n}\om})\bbI_B(S_n^\om)\big)dP(\om)
\end{eqnarray*}  
where $\bbI_B$ is the indicator function of the set $B$.
Then in both lattice and non-lattice
cases $U$ is a Radon measure on $\bbR$ so that $\int |g|dU<\infty$ for any
$C_{4\downarrow}(\bbR)$. Moreover if either $\mu_\om(h_\om'(0))$ or $\nu_\om'(0)(f_\om)$ do not depend on $\om$
then (\ref{renewal equations}) holds true for any function $g\in C_{4\downarrow}(\bbR)$.
\end{theorem}
In the proof of Lemma 7.2.1 in \cite{book} we showed that $\nu_\om(0)(h_\om'(0))=0$ and so we can always take $\mu_\om=\nu_\om(0)=\boldsymbol{\mu}_\om(0)$. Since $\nu_\om(z)\textbf{1}=0$ we can always take $f_\om\equiv\textbf{1}$ and then choose any $\mu_\om$  for this $f$. 
In the case when $\nu_\om'(0)(f_\om)$ does 
not depend on $\om$ the arguments in Section \ref{Renewal Section} are modified as follows. 
We first write 
\[
\bbE_P[\nu_\om(it)(f_\om)f_\rho^{\om}(t)]=\bbE_P[\mu_\om(h_\om(it))
f_{1,\rho}^\om(t)]
\]
where 
\[
f_{1,\rho}^\om(t)=\sum_{n\geq1}\rho^{n-1}\la_{\te^{-n}\om}(it)\nu_{\te^{-n}\om}(it)
(f_{\te^{-n}\om}).
\]
Therefore, for any function $g_1$ with the properties described at the beginning of
Section \ref{Renewal Section},
\[
V_\rho(g_1)=\int_{-\del_0}^{\del_0}e_1(t)\bbE_P[\mu_\om(h_\om(it))f_{1,\rho}^\om(t)]dt
\]
where in the non-lattice case $e_1(t)=\hat g_1(t)$, while in the lattice 
case $e_1(t)=\sum_{k}\hat g_1(t+\frac{2\pi k}{h})$.
Set $\tilde \Gam_n=\te^{-n}\Gam_n$. Repeating the arguments of the proof of Lemma \ref{VII.4} with 
$\la_{\om}(it)$, $f_{1,\rho}^{\om}(t)-f_{1,\rho}^{\te\om}(t)$,
$\la_{\te^{-n}\om,n}(it)$, $\tilde \Gam_n$ and 
$\nu_{\te^{-n}\om}(it)(f_{\te^{-n}\om})$ in place of
$\la_{\te^{-1}\om}(it)$, $f_\rho^\om(t)-f_{\rho}^{\te^{-1}\om}(t)$, 
$\la_{\om,n}(it)$, $\Gam_n$ and $\mu_{\te^n\om}(h_{\te^n\om}(it))$, respectively, 
we obtain (\ref{renewal equations}) in the case when $\nu_\om'(0)(f_\om)$ does not depend on $\om$.

\subsection{Markov chains with transition densities}\label{Densities}
Let $(\Om,\cF,P,\te)$ and $(\cX,\rho)$,   
$\cE\subset\Om\times \cX$ and  $\cE_\om$ satisfy the conditions specified in Section \ref{NonFib}.
For any $\om\in\Om$ denote by $B_\om$ the Banach space of all bounded
Borel functions $g:\cE_\om\to\bbC$ together with the supremum norm $\|\cdot\|_\infty$.
For any $g:\cE\to\bbC$ consider the functions $g_\om:\cE_\om\to\bbC$ given by 
$g_\om(x)=g(\om,x)$.
Then by Lemma 5.1.3 in \cite{book}, the norm $\om\to\|g_\om\|_\infty$ is a $\cF$-measurable
function of $\om$, for any measurable $g:\cE\to\bbC$.

Let $r_\om=r_\om(x,y):\cE_\om\times\cE_{\te\om}\to[0,\infty),\,\om\in\Om$ be a family of 
integrable in $y$ Borel measurable functions, $m_\om,\,\om\in\Om$ be a family of Borel probability measures
on $\cE_\om$ and $u:\cE\to\bbR$ be a measurable function so that $u_\om\in B_\om,\,$ $P$-a.s.   
and that the random variable $\sup|u_\om|=\|u_\om\|_\infty$ is bounded.
Consider the family of random operators $R_z^\om,\,z\in\bbC$ which map (bounded) 
Borel functions $g$ on $\cE_{\te\om}$ to Borel measurable functions on $\cE_\om$ 
by the formula
\begin{equation}\label{IntOp}
R^\om_zg(x)=\int_{\cE_{\te\om}} r_\om(x,y)e^{zu_{\te\om}(y)}g(y)dm_{\te\om}(y).
\end{equation} 
We will assume that 
$R_0^\om$ are Markov operators, namely that $R_0^\om \textbf{1}=\textbf{1}$
where $\textbf{1}$ is the function which takes the constant value $1$ on $\cE_{\te\om}$.
Observe that 
\[
\|R_0^\om\|_\infty:=
\sup_{g\in B_{\te\om}:\|g\|_\infty\leq 1}\|R_0^\om g\|_\infty=\|R_0^\om\textbf{1}\|_\infty
\]
and therefore for $P$-a.a. $\om$ 
we have $\|R_z^\om\|_\infty<\infty$ for any 
$z\in\bbC$, namely,  $R_z^\om$ is a continuous linear operator between the Banach spaces 
$B_{\te\om}$ and $B_\om$.

\begin{assumption}\label{Meas ass}
The maps $\om\to\int_{\cE_{\om}} g_{\om}(x)dm_{\te\om}(x)$ and 
$(\om,x)\to R_0^\om g_{\te\om}(x)$, $(\om,x)\in \cE$ are measurable for any bounded
measurable function $g:\cE\to\bbC$. 
\end{assumption}

For any $\om\in\Om$, $n\in\bbN$ and $z\in\bbC$ consider the
$n$-th order iterates $R_z^{\om,n}:B_{\te^n\om}\to B_\om$
given by 
\begin{equation}\label{Int Op oter 1}
R_z^{\om,n}=R_z^\om\circ R_z^{\te\om}\circ\cdots\circ R_z^{\te^{n-1}\om}.
\end{equation}
Then we can write
\[
R_0^{\om,n}g(x)=\int _{\cE_{\te^n\om}}r_\om(n,x,y)g(y)dm_{\te^n\om}(y)
\]
for some family $r_\om(n,\cdot,\cdot)=r_\om(n,x,y):\cE_\om\times\cE_{\te^n\om}\to[0,\infty)$
of integrable in $y$ Borel
measurable functions.
We will assume that the following random version of the two sided Doeblin condition holds true. 
\begin{assumption}\label{Doeb}
There exist a bounded random variable $j_\om\in\bbN$ 
and $\al_m(\om)\geq 1$, $m\in\bbN$ such that $P$-a.s.,
\begin{equation}\label{al j om}
\al_m(\om)\leq r_\om(m,x,y)\leq\big(\al_m(\om)\big)^{-1},
\end{equation}
for any $m\geq j_\om\,$, $x\in\cE_\om$ and $y\in\cE_{\te^m\om}$.
Moreover, let $j_0$ be so that $j_\om\leq j_0$, $P$-a.s. Then there exists $\al>0$
so that $\al_n(\om)\geq\al$ for any $j_0\leq n\leq 2j_0$.
\end{assumption}

Let $\mu_\om$ be a (measurable in $\om$) probability measure on $\cE_\om$ and let $\xi_n^{\te^n\om},\,n\geq1$
be a Markov chain with initial distribution $\mu_\om$ whose $n$-th step operator is given by 
$R_0^{\om,n}$. Set 
\[
S_n^\om=\sum_{j=0}^{n-1}u_{\te^j\om}(\xi_j^{\te^j\om}).
\]
Let $S_n$ be the random variable generated by drawing $\om$ according to $P$ and taking on the fibers
the distribution of $S_n^\om$, namely the random variable whose characteristic function is given by
\[
\bbE e^{itS_n}=\int\mu_\om(R_{it}^{\om,n}\textbf{1})dP(\om).
\]
Under Assumption \ref{Doeb}, in \cite{Kifer-1998} the author proved that the limit $\sig^2=\lim_{n\to\infty}n^{-1}\text{var}S_n^\om$ exists $P$-a.s., and it does not depend on $\om$.

Next, let $\om_0\in\Om$ and $n_0\in\bbN$ be so that $\te^{n_0}\om_0=\om_0$. We will call the case the non-lattice case if for any $t\in\bbR\setminus\{0\}$ the spectral radius of the operator $R_{it}^{\om_0,n_0}$ is strictly less than $1$. We will call 
the case a lattice one if for some $h>0$ the function $u$ takes values on the lattice $h\bbZ$ and 
the spectral radius of the operators 
$R_{it}^{\om_0,n_0},\,t\in(-\frac{2\pi}h,\frac{2\pi}h)\setminus\{0\}$ are strictly less than $1$. We refer to \cite{HenHerve} for a characterization 
of these lattice and non-lattice cases which resembles the description of these cases in the transfer operator case.
\begin{theorem}
Suppose that Assumptions \ref{Meas ass} and \ref{Doeb} hold true, that $\sig^2>0$ and that $\gam=\mu_\om(u_\om)$ does not depend on $\om$. Then $\sig^{-1}n^{-\frac12}(S_n-n\gam),\,n\geq1$ converges  in distribution as $n\to\infty$ towards the standard normal law, and $S_n-n\gam,\,n\geq1$ satisfies the appropriate LLT (in both lattice and non-lattice cases). Moreover,  when $\gam>0$ then
all the statements in Theorem \ref{renewa2} hold true.
\end{theorem}

\begin{remark}
In the above integral operator case 
it is possible to obtain similar limit theorems without using Assumption \ref{PerPointAss} 
and (\ref{Neighb mix}), relying instead on some assumption on the distribution
of the process $j_{\te^k\om},\,k\geq 1$ in the spirit of (\ref{V-mix}).
In \cite{HKllt} we proved a local limit theorem for certain ``nonconventional" sums. Our proof
there was based on a certain reduction to a problem of bounding expectations of norms of 
iterates of random Fourier operators (the proof was in the spirit of the argument in 
\cite{Neg2}). This is exactly the situation of annealed limit theorems, and so, similar
to \cite{HKllt} argument will yield the desired results.
\end{remark}

\end{document}